\documentclass[12pt]{article}

\usepackage{amssymb}
\usepackage{amsmath,amsfonts}
\usepackage{amsthm,bbm,eufrak,upgreek}
\usepackage{mathrsfs,booktabs}
\usepackage{amsxtra,mathtools}
\usepackage[colorlinks=True,linkcolor=blue,anchorcolor=blue,citecolor=blue,filecolor=blue,CJKbookmarks=True]{hyperref}
\usepackage[a4paper,left=2.6cm,right=2.6cm,top=3.2cm,bottom=3.2cm]{geometry}

\numberwithin{equation}{section}

\newcommand{\Z}{{\mathbb Z}}
\newcommand{\C}{{\mathbb C}}

\newcommand{\be}{\beta}

\newcommand{\gm}{\gamma}

\newcommand{\la}{\langle}
\newcommand{\ra}{\rangle}

\newcommand{\wo}{\widehat{\frak{so}_m}}

\DeclareMathOperator{\Aut}{Aut}

\DeclareMathOperator{\End}{End}

\DeclareMathOperator{\wt}{wt}
\DeclareMathOperator{\Hom}{Hom}
\DeclareMathOperator{\Irr}{Irr}

\newtheorem{thm}{Theorem}[section]
\newtheorem{prop}[thm]{Proposition}
\newtheorem{lem}[thm]{Lemma}

\newtheorem{rmk}[thm]{Remark}
\newtheorem{defn}[thm]{Definition}

\begin{document}

\begin{center}
{\Large \bf  Representations of the orbifold VOAS $L_{\widehat{\frak{sl}_2}}(k,0)^{K}$ and the commutant VOAS $C_{{L_{\widehat{\mathfrak{so}_m}}(1,0)}^{\otimes 3}}({L_{\widehat{\mathfrak{so}_m}}(3,0)})$ }
\end{center}

\begin{center}
{Cuipo Jiang
\footnote{Supported by China NSF grants No.11771281 and No.11531004. email: cpjiang@sjtu.edu.cn.}
and Bing Wang
\footnote{Supported by China NSF grant No.11771281. email: ering123@sjtu.edu.cn.}\\
School of Mathematical  Sciences, Shanghai Jiao Tong University\\
Shanghai 200240, China}
\end{center}

\begin{abstract}
For the Klein group $K$, $k\in\mathbb{Z}_{\geqslant 1}$ and $m\in\mathbb{Z}_{\geqslant 4}$, we study the representations of the orbifold vertex operator algebra $L_{\widehat{\mathfrak{sl}_2}}(k,0)^{K}$ and the commutant vertex operator algebra of $L_{\widehat{\mathfrak{so}_m}}(3,0)$ in $L_{\widehat{\mathfrak{so}_m}}(1,0)^{\otimes 3}$ which can be realized as the orbifold vertex operator subalgebra $L_{\widehat{\mathfrak{sl}_2}}(2m,0)^{K}$ or its extension. All the irreducible modules for $L_{\widehat{\mathfrak{sl}_2}}(k,0)^{K}$ and $C_{{L_{\widehat{\mathfrak{so}_m}}(1,0)}^{\otimes 3}}({L_{\widehat{\mathfrak{so}_m}}(3,0)})$ are classified and constructed explicitly.
\end{abstract}


\section{Introduction}

The coset construction, initiated in \cite{GKO85} and \cite{GKO86}, is one of the major ways to construct new vertex algebras from given ones. And the commutant construction was later introduced  in \cite{FZ92} from the point view of vertex operator algebras. Given a vertex algebra $V$ and a subalgebra $U\subseteq V$, it has been proved in \cite{FZ92} that under suitable conditions $C_V(U)$, which is called the commutant of $U$ in $V$,  is a vertex subalgebra of $V$. Since then, describing commutant vertex operator algebras has been one of the most interesting questions in the theory of vertex operator algebras.
The orbifold construction is another important way to construct new vertex algebras from existing ones. From the viewpoint of representations, the orbifold theory is the representation theory for the fixed point vertex operator subalgebra of a vertex operator algebra under the action of a finite automorphism group. Let $V$ be a vertex operator algebra and $G$ a finite group consisting of certain automorphisms of $V$, the fixed point subalgebra $V^G=\{v \in V \mid gv=v, g \in G\}$ is called an orbifold vertex operator subalgebra of $V$.

This paper is prompted by the  results of \cite{JLam19}. Let $\mathfrak{g}$ be a complex finite dimensional simple Lie algebra, and  $L_{\hat{\mathfrak{g}}}(1,0)$ the associated rational and simple affine vertex operator algebra  \cite{DLM97-1}, \cite{FZ92}, \cite{LeL04}, \cite{Li96-1}. For $\ell \in \mathbb{Z}_+$, the tensor product $L_{\hat{\mathfrak{g}}}(1,0)^{\otimes \ell}$ is still rational and the diagonal action of $\hat{\mathfrak{g}}$ on $L_{\hat{\mathfrak{g}}}(1,0)^{\otimes \ell}$ defines a vertex subalgebra $L_{\hat{\mathfrak{g}}}(\ell,0)$ of level $\ell$.
For the simple Lie algebra $\mathfrak{so}_{m}$, the commutant vertex operator algebras $C_{{L_{\widehat{\mathfrak{so}_m}}(1,0)}^{\otimes \ell}}({L_{\widehat{\mathfrak{so}_m}}(\ell,0)})$ for $m\geqslant 4$, $\ell \geqslant 3$ were investigated in \cite{JLam19}.
In particular, the commutant vertex operator algebra $C_{{L_{\widehat{\mathfrak{so}_m}}(1,0)}^{\otimes 3}}({L_{\widehat{\mathfrak{so}_m}}(3,0)})$ can be realized as the obifold vertex operator algebra $L_{\widehat{\mathfrak{sl}_2}}(2m,0)^{K}$ if $m$ is odd and  $(L_{\widehat{\mathfrak{sl}_2}}(2m,0)+L_{\widehat{\mathfrak{sl}_2}}(2m,2m))^{K}$ if $m$ is even, where $K\leqslant\Aut(L_{\widehat{\mathfrak{sl}_2}}(2m,0))$ is the abelian Klein group of order $4$.
In this paper, we classify and construct all the irreducible modules for both the orbifold vertex operator algebras $L_{\widehat{\mathfrak{sl}_2}}(k,0)^{K}$ for $k\geqslant 1$ and  the commutant vertex operator algebras $C_{{L_{\widehat{\mathfrak{so}_m}}(1,0)}^{\otimes 3}}({L_{\widehat{\mathfrak{so}_m}}(3,0)})$ for $m\geqslant 4$. Note that for $m=3$,  $C_{{L_{\widehat{\mathfrak{so}_m}}(1,0)}^{\otimes 3}}({L_{\widehat{\mathfrak{so}_m}}(3,0)})$  has been completely studied in \cite{CL}, \cite{DJL}, \cite{DLY}, \cite{JLin16}, \cite{KLY}, \cite{KMY}.

 It is well known that $L_{\widehat{\mathfrak{sl}_2}}(k,0)$ is a regular and selfdual vertex operator algebra of CFT type for $k\in\Z_{+}$ \cite{FZ92}, \cite{LeL04}. Let $G$ be a finite solvable automorphism group of $L_{\widehat{\mathfrak{sl}_2}}(k,0)$, then $L_{\widehat{\mathfrak{sl}_2}}(k,0)^G$ is also a regular and selfdual vertex operator algebra of CFT type \cite{CM}, \cite{M15}. It follows from \cite{DRX17} that any irreducible $L_{\widehat{\mathfrak{sl}_2}}(k,0)^G$-module occurs in an irreducible $g$-twisted $L_{\widehat{\mathfrak{sl}_2}}(k,0)$-module for some $g \in G$.
  Let $G=K$ be the Klein subgroup of $\Aut(L_{\widehat{\mathfrak{sl}_2}}(k,0))$. We first construct twisted $\sigma$-modules of $L_{\widehat{\mathfrak{sl}_2}}(k,0)$ for every $\sigma\in K$. Then by computing the associated twisted algebras introduced in \cite{DRX17}, we
   classify and construct all the irreducible modules for  the orbifold vertex operator algebra $L_{\widehat{\mathfrak{sl}_2}}(k,0)^{K}$, $k\in\Z_{+}$. It turns out
   that there are $\frac{11(k+1)}{2}$ inequivalent irreducible $L_{\widehat{\mathfrak{sl}_2}}(k,0)^K$-modules if $k$ is odd and $\frac{11k+32}{2}$ inequivalent irreducible $L_{\widehat{\mathfrak{sl}_2}}(k,0)^{K}$-modules if $k$ is even. We then deduce that  $C_{{L_{\widehat{\mathfrak{so}_m}}(1,0)}^{\otimes 3}}({L_{\widehat{\mathfrak{so}_m}}(3,0)})$ has $11m+16$ inequivalent irreducible modules  for  $m\in 2\Z_{\geqslant 2}+1$, and has $8m+32$ inequivalent irreducible modules if $m\in 2\Z_{\geqslant 2}$.

 The paper is organized as follows. In Section 2, we briefly review some basic notations and facts on vertex operator algebras. In Section 3, we first give the action of the Klein group $K$ on $L_{\widehat{\mathfrak{sl}_{2}}}(k,0)$ and realize each element of $K$ as an inner automorphism of $\mathfrak{sl}_2$. We next review the construction of the affine vertex operator algebra $L_{\widehat{\mathfrak{sl}_2}}(2m,0)$ by using the fermionic vertex superalgebras. In Section 4, we classify and construct all the irreducible modules of the orbifold vertex operator algebras $L_{\widehat{\mathfrak{sl}_2}}(k,0)^{K}$ for $k\geqslant 1$ and the commutant vertex operator algebras $C_{{L_{\widehat{\mathfrak{so}_m}}(1,0)}^{\otimes 3}}({L_{\widehat{\mathfrak{so}_m}}(3,0)})$ for $m\geqslant 4$.

We use the usual symbols $\C$ for the complex numbers, $\Z$ for the integers, and $\Z_{+}$ for the positive integers.

\section{Preliminary}

Let $(V,Y,\mathbbm{1},\omega)$ be a vertex operator algebra \cite{Bor86}, \cite{FLM88}. We first review basics from \cite{DLM98-1}, \cite{DLM00}, \cite{FHL93} and \cite{LeL04}. Let $g$ be an automorphism of the vertex operator algebra $V$ of finite order $T$. Denote the decomposition of $V$ into eigenspaces of $g$ as:
\[ V=\bigoplus_{r\in \mathbb{Z}/T\mathbb{Z}}V^r,\]
where $V^r= \{v \in V |gv = e^{-2\pi \sqrt{-1}\frac{r}{T}}v\}$, $0 \leqslant r \leqslant T-1$. We use $r$ to denote both an integer between $0$ and $T-1$ and its residue class modulo $T$ in this situation.

\begin{defn}

Let $V$ be a vertex operator algebra. A weak $g$-twisted $V$-module is a vector space $M$ equipped with a linear map
\begin{align*}
Y_M(\cdot, x) : V & \longrightarrow (\End M)[[x^{\frac{1}{T}},x^{-\frac{1}{T}}]] \\
                v & \longmapsto Y_M(v,x) = \sum_{n\in \frac{1}{T}\mathbb{Z}}v_nx^{-n-1},
\end{align*}
 where $v_n \in \End M $, satisfying the following conditions for $0 \leqslant r \leqslant T-1$, $u \in V^r, v \in V$, $w \in M$:
\[Y_M(u,x) = \sum_{n\in \frac{r}{T}+\mathbb{Z}}u_nx^{-n-1},\]
\[u_sw=0 \quad\text{for} \quad s \gg 0,\]
\[Y_M(\mathbbm{1},x) = id_M,\]
\[x_{0}^{-1}\delta(\frac{x_1-x_2}{x_0})Y_M(u,x_1)Y_M(v,x_2)-x_{0}^{-1}\delta(\frac{x_2-x_1}{-x_0})Y_M(v,x_2)Y_M(u,x_1)\]
\[=x_{2}^{-1}(\frac{x_1-x_0}{x_2})^{-\frac{r}{T}}\delta(\frac{x_1-x_0}{x_2})Y_M(Y(u,x_0)v,x_2),\]
where $\delta(x)=\sum_{n\in\mathbb{Z}}x^n$ and all binomial expressions are to be expanded in nonnegative integral powers of the second variable.
\end{defn}

The following Borcherds identities can be derived from the twisted-Jacobi identity \cite{DLM98-1}, \cite{XX98}.

\begin{equation}
[ u_{m + \frac{r}{T}}, v_{n + \frac{s}{T}} ] = \sum_{i=0}^{\infty} \binom{m + \frac{r}{T}}{i}(u_iv)_{m + n + \frac{r+s}{T}-i},
\label{Borcherds identity 1}
\end{equation}

\begin{equation}
\sum_{i=0}^{\infty} \binom{\frac{r}{T}}{i}(u_{m+i}v)_{n+\frac{r+s}{T}-i} = \sum_{i=0}^{\infty} (-1)^i \binom{m}{i} (u_{m+\frac{r}{T}-i}v_{n+\frac{s}{T}+i} -(-1)^mv_{m+n+\frac{s}{T}-i}u_{\frac{r}{T}+i}),
\label{Borcherds identity 2}
\end{equation}
where $u \in V^r$, $v \in V^s$, $m$, $n \in \mathbb{Z}$.

\begin{defn}
An admissible $g$-twisted $V$-module is a weak $g$-twisted $V$-module which carries a $\frac{1}{T}\mathbb{Z}_+$-grading $M = \oplus_{n\in \frac{1}{T}\mathbb{Z}_+}M(n)$ satisfying $v_mM(n) \subseteq M(n+r-m-1)$ for homogeneous $v \in V_r$, $m$, $n \in\frac{1}{T}\mathbb{Z}$.
\end{defn}

\begin{defn}
A $g$-twisted $V$-module is a weak $g$-twisted $V$-module which carries a $\mathbb{C}$-grading:
\[M=\oplus_{ \lambda \in \mathbb{C}}M_\lambda,\]
such that dim $M_{\lambda} < \infty$, $M_{\lambda +\frac{n}{T}} = 0$ for fixed $\lambda$ and $n \ll 0$, $L(0)w = \lambda w = ( \wt w )w$ for $w \in M_{\lambda}$, where $L(0)$ is the component operator of $Y_M(\omega, x) = \sum_{n\in \mathbb{Z}}L(n)x^{-n-2}$.
\end{defn}

\begin{rmk}
If $g = id_V$, we have the notations of weak, admissible and ordinary $V$-modules \cite{DLM97-1}.
\end{rmk}

\begin{defn}
A vertex operator algebra $V$ is called $g$-rational if the admissible $g$-twisted $V$-module category is semisimple. $V$ is called rational if $V$ is $id_V$-rational.
\end{defn}

If $M = \oplus_{n\in \frac{1}{T}\mathbb{Z}_+}M(n)$ is an admissible $g$-twisted $V$-module, the contragredient module $M'$ is defined as follows:
\[ M' = \oplus_{n\in \frac{1}{T}\mathbb{Z}_+}M(n)^{*},\]
where $M(n)^{*} = \rm{Hom}_{\mathbb{C}}(M(n), \mathbb{C})$. The vertex operator $Y_{M'}(a, z)$ is defined for $a \in V$ via
\[ \langle Y_{M'}(a, z)f, u \rangle = \langle f, Y_M(e^{zL(1)}(-z^{-2})^{L(0)}a, z^{-1})u \rangle, \]
where $\langle f, u \rangle = f(u)$ is the natural pairing $M' \times M \to \mathbb{C}$. It follows from \cite{FHL93} and \cite{XF00} that $( M' , Y_{M'} )$ is an admissible $g^{-1}$-twisted $V$-module. We can also define the contragredient module $M'$ for a $g$-twisted $V$-module $M$. In this case, $M'$ is a $g^{-1}$-twisted $V$-module. Moreover, $M$ is irreducible if and only if $M'$ is irreducible. $M$ is said to be selfdual if $M$ is $V$-isomorphic to $M'$. In particular, $V$ is said to be a selfdual vertex operator algebra if $V$ is isomorphic to $V'$.
We recall the following  concept from \cite{Zhu96}.
\begin{defn}
A vertex operator algebra is called $C_2$-cofinite if $C_2(V)$ has finite codimension (i.e., dim $V/C_2(V) < \infty$),
where $C_2(V) = \langle u_{-2}v \mid u, v \in V \rangle$.
\end{defn}

We have the following result from \cite{ABD04}, \cite{DLM98-1} and \cite{Zhu96}.

\begin{thm}
If $V$ is a vertex operator algebra satisfying the $C_2$-cofinite property, then $V$ has only finitely many irreducible admissible modules up to isomorphism. The rationality of $V$ also implies the same result.
\end{thm}

We have the following results from \cite{DLM98-1} and \cite{DLM00}.

\begin{thm}
If V is $g$-rational vertex operator algebra, then

(1) Any irreducible admissible $g$-twisted $V$-module $M$ is a $g$-twisted $V$-module. Moreover, there exists a number $\lambda \in \mathbb{C}$ such that $M=\oplus_{n\in \frac{1}{T}\mathbb{Z}_+}M_{\lambda + n}$, where $M_{\lambda} \ne 0$. The number $\lambda$ is called the conformal weight of $M$;

(2)There are only finitely many irreducible admissible $g$-twisted $V$-modules up to isomorphism.
\end{thm}

\begin{defn}
A vertex operator algebra $V$ is called regular if every weak $V$-module is a direct sum of irreducible $V$-modules, i.e., the weak module category is semisimple.
\end{defn}

\begin{defn}
A vertex operator algebra $V=\oplus_{n\in \mathbb{Z}}V_n$ is said to be of CFT type if $V_n = 0$ for $n < 0$ and $V_0 = \mathbb{C}\mathbbm{1}$.
\end{defn}

\begin{rmk}
It is proved in \cite{ABD04}  that for a  CFT type vertex operator algebra $V$,  regularity is equivalent to rationality and $C_2$-cofiniteness.
\end{rmk}

Next, we recall the notation of the commutant vertex operator algebra following \cite{FZ92} and \cite{LeL04}.

Let $(V, Y, \mathbbm{1}, \omega)$ be a vertex operator algebra of central charge $c_V$ and $(U, Y, \mathbbm{1}, \omega')$ a vertex operator subalgebra of $V$ of central charge $c_U$. Set
\[C_V(U) = \{ v\in V \mid  [Y(u,x_1), Y(v,x_2)] = 0,  \forall  u \in U \}.\]
If $\omega' \in U\cap V_2$ and $L(1)\omega' = 0$, then $(C_V(U), Y, \mathbbm{1}, \omega - \omega')$ is a vertex operator subalgebra of $V$ of central charge $c_V-c_U$ which is called the \emph{commutant} of $U$ in $V$.

 We now review some notations and  facts about the action of the automorphism group on twisted modules of vertex operator algebra $V$ from \cite{DLM00}, \cite{DRX17}, \cite{DY02}, \cite{MT04}.

Let $g, h$ be two automorphisms of $V$. If $(M, Y_M)$ is a weak $g$-twisted $V$-module, there is a weak $h^{-1}gh$-twisted $V$-module $(M \circ h, Y_{M \circ h})$ where $M \circ h = M$ as vector spaces and $Y_{M \circ h}(v, z) = Y_M(hv, z)$ for $v \in V$. This gives a right action of $\Aut(V)$ on weak twisted $V$-modules. Symbolically, we write
\[  (M, Y_M) \circ h = (M \circ h,Y_{M \circ h}) = M \circ h. \]
The $V$-module $M$ is called \emph{$h$-stable} if $M \circ h$ and $M$ are isomorphic $V$-modules.

Let $G$ be a finite group of automorphisms of $V$, $g \in G$ of finite order $T$ and $M = (M, Y_M)$ an irreducible $g$-twisted $V$-module. Define a subgroup $G_M$ of $G$ consisting all of $h \in G$ such that $M$ is $h$-stable. For $h \in G_{M}$, there is a linear isomorphism $\phi(h) : M \to M$ satisfying
\[ \phi(h)Y_M(v, z){\phi(h)}^{-1} = Y_{M \circ h}(v, z) = Y_M(hv, z)  \]
for $v \in V$. The simplicity of $M$ together with Schur's lemma shows that $h \mapsto \phi(h)$ is a projective representation of $G_M$ on $M$. Let $\alpha_M$ be the corresponding $2$-cocycle in $C^2(G, \mathbb{C}^*)$. Then $M$ is a module for the twisted group algebra $\mathbb{C}^{\alpha_M}[G_M]$ which is  a semisimple associative algebra. A basic fact is that $g$ belongs to $G_M$. Let $M^r=\oplus_{n\in \frac{r}{T}+\mathbb{Z}_+}M(n)$ for $r = 0$, $1$, $\cdots$, $T-1$, then $M = \oplus_{n\in \frac{1}{T}\mathbb{Z}_+}M(n) = \oplus_{r=0}^{T-1}M^r$ and each $M^r$ is an irreducible $V^{\langle g \rangle}$-module on which  $\phi(g)$ acts as constant $e^{2\pi \sqrt{-1}\frac{r}{T}}$ \cite{DM97}, \cite{DRX17}.

Let $\Lambda_{G_M, \alpha_M}$ be the set of all irreducible characters $\lambda$ of $\mathbb{C}^{\alpha_M}[G_M]$. Then
\begin{equation}
M = \oplus_{\lambda \in \Lambda_{G_M, \alpha_M}}W_{\lambda} \otimes M_{\lambda},
\end{equation}
where $W_{\lambda}$ is the simple $\mathbb{C}^{\alpha_M}[G_M]$-module affording $\lambda$ and $M_{\lambda} = \Hom_{\mathbb{C}^{\alpha_M}[G_M]}(W_{\lambda}, M)$ is the mulitiplicity of $W_{\lambda}$ in $M$. And each $M_{\lambda}$ is a module for the vertex operator subalgebra $V^{G_M}$.

The following results follow from \cite{DRX17} and \cite{DY02}.

\begin{thm}
\label{DRX17 thm1}
With the same notations as above we have

(1) $W_{\lambda} \otimes M_{\lambda}$ is nonzero for any $\lambda \in \Lambda_{G_M, \alpha_M}$.

(2) Each $M_{\lambda}$ is an irreducible $V^{G_M}$-module.

(3) $M_{\lambda}$ and $M_{\mu}$ are equivalent $V^{G_M}$-module if and only if $\lambda = \mu$.
\end{thm}

\begin{thm}
\label{DRX17 thm2}
Let $g$, $h \in G$, $M$ be an irreducible $g$-twisted $V$-module, and $N$ an irreducible $h$-twisted $V$-module. If $M$, $N$ are not in the same orbit under the action of $G$, then the irreducible $V^G$-modules $M_{\lambda}$ and $N_{\mu}$ are inequivalent for any $\lambda \in \Lambda_{G_M, \alpha_M}$ and $\mu \in \Lambda_{G_N, \alpha_N}$.
\end{thm}

\begin{thm}
Let $V^G$ be a regular and selfdual vertex operator algebra of CFT type. Then any irreducible $V^G$-module is isomorphic to $M_\lambda$ for some irreducible $g$-twisted $V$-module $M$ and some $\lambda \in \Lambda_{G_M, \alpha_M}$. In particular, if $V$ is a regular and selfdual vertex operator algebra of CFT type and $G$ is solvable, then any irreducible $V^G$-module is isomorphic to some $M_\lambda$.
\label{DRX17 thm3}
\end{thm}

We now recall from \cite{FHL93} the notions of intertwining operators and fusion rules.

\begin{defn} Let $(V,\ Y)$ be a vertex operator algebra and
let $(W^{1},\ Y^{1}),\ (W^{2},\ Y^{2})$ and $(W^{3},\ Y^{3})$ be
$V$-modules. An \emph{intertwining operator} of type $\left(\begin{array}{c}
W^{3}\\
W^{1\ }W^{2}
\end{array}\right)$ is a linear map
\begin{align*}
I(\cdot, z) : W^{1} & \longrightarrow \Hom ( W^{2}, W^{3} ) \{ z \} \\
                 u  & \longmapsto I(u, z) = \sum_{n\in\mathbb{Q}}u_{n}z^{-n-1}
\end{align*}
satisfying:

(1) for any $u\in W^{1}$ and $v\in W^{2}$, $u_{n}v=0$ for $n$
sufficiently large;

(2) $I(L(-1)v,\ z)=\frac{d}{dz}I(v,\ z)$;

(3) (Jacobi identity) for any $u\in V,\ v\in W^{1}$

\[
z_{0}^{-1}\delta\left(\frac{z_{1}-z_{2}}{z_{0}}\right)Y^{3}(u,\ z_{1})I(v,\ z_{2})-z_{0}^{-1}\delta\left(\frac{-z_{2}+z_{1}}{z_{0}}\right)I(v,\ z_{2})Y^{2}(u,\ z_{1})
\]
\[
=z_{2}^{-1}\delta\left(\frac{z_{1}-z_{0}}{z_{2}}\right)I(Y^{1}(u,\ z_{0})v,\ z_{2}).
\]

The space of all intertwining operators of type $\left(\begin{array}{c}
W^{3}\\
W^{1}\ W^{2}
\end{array}\right)$ is denoted by
$$I_{V}\left(\begin{array}{c}
W^{3}\\
W^{1}\ W^{2}
\end{array}\right).$$ Let $N_{W^{1},\ W^{2}}^{W^{3}}=\dim I_{V}\left(\begin{array}{c}
W^{3}\\
W^{1}\ W^{2}
\end{array}\right)$. These integers $N_{W^{1},\ W^{2}}^{W^{3}}$ are usually called the
\emph{fusion rules}.
\end{defn}




\begin{defn} Let $V$ be a vertex operator algebra, and $W^{1},$
$W^{2}$ be two $V$-modules. A module $(W,I)$, where $I\in I_{V}\left(\begin{array}{c}
\ \ W\ \\
W^{1}\ \ W^{2}
\end{array}\right),$ is called a \emph{tensor product} (or fusion product) of $W^{1}$
and $W^{2}$ if for any $V$-module $M$ and $\mathcal{Y}\in I_{V}\left(\begin{array}{c}
\ \ M\ \\
W^{1}\ \ W^{2}
\end{array}\right),$ there is a unique $V$-module homomorphism $f:W\rightarrow M,$ such
that $\mathcal{Y}=f\circ I.$ As usual, we denote $(W,I)$ by $W^{1}\boxtimes_{V}W^{2}$ or $W^{1}\boxtimes W^{2}$ simply.
\end{defn}

\begin{rmk}
It is well known that if $V$ is rational, then for any two irreducible
$V$-modules $W^{1}$ and $W^{2},$ the fusion product $W^{1}\boxtimes_{V}W^{2}$ exists and
$$
W^{1}\boxtimes_{V}W^{2}=\sum_{W}N_{W^{1},\ W^{2}}^{W}W,
$$
 where $W$ runs over the set of equivalence classes of irreducible
$V$-modules.
\end{rmk}
Fusion rules have the following symmetric property \cite{FHL93}.

\begin{prop}\label{fusionsymm.}
Let $W^{i} (i=1,2,3)$ be $V$-modules. Then
$$N_{W^{1},W^{2}}^{W^{3}}=N_{W^{2},W^{1}}^{W^{3}}, \ N_{W^{1},W^{2}}^{W^{3}}=N_{W^{1},(W^{3})^{'}}^{(W^{2})^{'}}.$$
\end{prop}
\begin{defn} Let $V$ be a simple vertex operator algebra, a simple $V$-module $M$ is called a simple current if for any irreducible $V$-module, $M\boxtimes W$ exists and is also an irreducible $V$-module.
\end{defn}
We have the following fact about simple current extension from \cite{DLM96-1}, \cite{Lam01} and \cite{Li97-1}.
\begin{thm}\label{sc}
Let $V$ be a simple rational and $C_2$-cofinite  vertex operator algebra. Let $M$ be a simple current module of $V$ such that $\widetilde{V}=V\oplus M$ has a vertex operator algebra structure. Then, as a $V$-module, any irreducible module of $\widetilde{V}$ is of the following form:

(1) $N\oplus \widetilde{N}$, where $N$ is an irreducible $V$-module and $\widetilde{N}=M\boxtimes_{V}N\ncong N$. In this case, $N\oplus \widetilde{N}$ has a unique $\widetilde{V}$-module structure and the weight of  $\widetilde{N}$ is congruent to the weight of $N$ {\rm mod}$\Z$.

(2) $N$, where  $N$ is an irreducible $V$-module such that ${N} \cong M\boxtimes_{V}N$. In this case, there exist two non-isomorphic $\widetilde{V}$-module structures on $N$. These two modules are denoted by $N^+$ and $N^-$, respectively.
The modules in case (1) are often called ``nonsplit'' type and the modules in case (2) are often called ``split'' type.
\end{thm}

Let $V$, $\widetilde{V}$, $M$ be as in Theorem \ref{sc}. Assume that $N$ is an irreducible $V$-module such that $\widetilde{N}=M\boxtimes_{V}N\cong N$. Then there exists a $V$ isomorphism $f$ from $N$ to $\widetilde{N}$ such that
$$
fY(a,z)f(u)=Y(a,z)u,
$$
for $a\in M$, $u\in N$. Then \cite{Li97-1}
\begin{equation}\label{esc}
N^+=\{w+f(w)|w\in N\}, \ N^-=\{w-f(w)|w\in N\},
\end{equation}
and $N\oplus \widetilde{N}=N^{+}\oplus N^{-}$.

\section{$L_{\widehat{\frak{sl}_{2}}}(k,0)^{K}$ and Fermionic construction of $L_{\widehat{\mathfrak{sl}_2}}(2m,0)$}

In this section, we will   introduce the Klein group ${K}$ which is a subgroup of $\Aut(L_{\widehat{\frak{sl}_{2}}}(k,0))$, and realize each element of ${K}$ as an inner automorphism of $\frak{sl}_2$. We will also review the construction of the simple vertex operator algebra $L_{\widehat{\mathfrak{sl}_2}}(2m,0)$ by using the fermionic vertex superalgebras following \cite{AP}, \cite{JLam19} and \cite{KW94}.

Let $h,e,f$ be a standard basis of $\frak{sl}_2(\C)$, define automorphisms $\sigma_1$ and $\sigma_2$ of $\frak{sl}_2(\C)$ as follows:
$$
\sigma_1(h)=h, \ \sigma_1(e)=-e, \ \sigma_1(f)=-f;
$$
$$
\sigma_2(h)=-h, \ \sigma_2(e)=f, \ \sigma_2(f)=e.
$$
It is obvious that the automorphic subgroup generated by $\sigma_1$ and $\sigma_2$ is isomorphic to the Klein group ${K}$, and  ${K}$ can be lifted to an automorphic subgroup of the vertex operator algebra $L_{\widehat{\mathfrak{sl}_{2}}}(k,0)$ for $k \in \Z_{+}$. Set
\[h^{(1)} = h, \quad e^{(1)} = e, \quad f^{(1)} = f,\]
\[h^{(2)} = e + f, \quad e^{(2)} = \frac{1}{2}(h - e + f), \quad f^{(2)} = \frac{1}{2}(h + e - f),\]
\[h^{(3)} = \sqrt{-1}(e - f), \quad e^{(3)} = \frac{1}{2}(\sqrt{-1}h + e + f), \quad f^{(3)} = \frac{1}{2}(-\sqrt{-1}h + e + f).\]
We can verify that $\{ h^{(1)}$, $e^{(1)}$, $f^{(1)} \}$, $\{ h^{(2)}$, $e^{(2)}$, $f^{(2)} \}$ and $\{ h^{(3)}$, $e^{(3)}$, $f^{(3)} \}$ are $\mathfrak{sl}_2$-triples, and for $r = 1$, $2$, or $3$, we have
\[\sigma_r(h^{(r)}) = h^{(r)}, \quad \sigma_r(e^{(r)}) = -e^{(r)}, \quad \sigma_r(f^{(r)}) = -f^{(r)}.\]
Let $h^{(r)'}=\frac{1}{4}h^{(r)}$ for  $r = 1$, $2$, $3$. Then we have the following result.

\begin{prop}
For  $r = 1$, $2$, $3$, $e^{2\pi \sqrt{-1}h^{(r)'}(0)} = \sigma_r$.
\end{prop}
\begin{proof}
For each $r = 1$, $2$, $3$,  direct calculations yield that
\[ L(n)h^{(r)'} = \delta_{n,0}h^{(r)'}, \quad h^{(r)'}(n)h^{(r)'} = \frac{1}{8}\delta_{n,1}k,\quad \text{for} \quad n \in \mathbb{Z}_+ , \]
\[ h^{(r)'}(0)e^{(r)} = \frac{1}{2}e^{(r)}, \quad h^{(r)'}(0)f^{(r)} = -\frac{1}{2}f^{(r)}, \quad h^{(r)'}(0)h^{(r)'} = 0, \quad e^{(r)}(0)f^{(r)} = 4h^{(r)'} ,\]
where $L(n)=\omega(n+1)$, $\omega$ is the conformal vector of $L_{\widehat{\mathfrak{sl}_2}}(k,0)$. These equations show that $h^{(r)'}(0)$ acts on $L_{\widehat{\mathfrak{sl}_2}}(k,0)$ semisimply with rational eigenvalues. From \cite{Li96-2}, we know that $e^{2\pi \sqrt{-1}h^{(r)'}(0)}$ is an automorphism of $L_{\widehat{\mathfrak{sl}_2}}(k,0)$. Moreover, $e^{2\pi \sqrt{-1}h^{(r)'}(0)}(h^{(r)}) = h^{(r)}$, $e^{2\pi \sqrt{-1}h^{(r)'}(0)}(e^{(r)}) = -e^{(r)}$, $e^{2\pi \sqrt{-1}h^{(r)'}(0)}(f^{(r)}) = -f^{(r)}$. Thus $e^{2\pi \sqrt{-1}h^{(r)'}(0)} = \sigma_r$.
\end{proof}

For a positive integer  $m\geqslant 4$,  let $\mathcal{C}\mathbf{\ell}_{3m}$  be the Clifford algebra generated by
\[\psi_{ij}(r),\quad r \in \frac{1}{2} + \mathbb{Z}, 1\leqslant i\leqslant m, 1\leqslant j\leqslant 3,\]
with the non-trivial relations
\[[\psi_{ij}(r), \; \psi_{tl}(s)]_+ = \delta_{it}\delta_{jl}\delta_{r+s,0} \;,\]
where $1\leqslant i, t\leqslant m$, $1\leqslant j, l\leqslant 3$, $r, s \in \frac{1}{2} + \mathbb{Z}$.

Let $\mathcal{F}_{3m}$ be the irreducible $\mathcal{C}\mathbf{\ell}_{3m}$-module generated by the cyclic vector $\mathbbm{1}$ such that
\[\psi_{ij}(r)\mathbbm{1}=0, \quad \text{for} \quad r > 0, 1\leqslant i\leqslant m, 1\leqslant j\leqslant 3. \]
Define the following fields on $\mathcal{F}_{3m}$ as follows
\[ \psi_{ij}(x)= \sum_{r\in \mathbb{Z}}\psi_{ij}(r+\frac{1}{2})x^{-r-1}.  \]
Then the fields $\psi_{ij}(x)$, $1\leqslant i\leqslant m$, $1\leqslant j\leqslant 3$ generate the unique structure of a simple vertex superalgebra on $\mathcal{F}_{3m}$. Let $\mathcal{F}_{3m}^{even}$ be the even part of the vertex superalgebra $\mathcal{F}_{3m}$. It was proved in \cite{FF85} that for $m\geqslant 2$,
\[\mathcal{F}_{3m}^{even}\cong L_{\widehat{\mathfrak{so}_{3m}}}(1,0).\]

Obviously, the vertex operator algebra $L_{\widehat{\mathfrak{so}_{3m}}}(1,0)$ is generated by $\psi_{ri}(-\frac{1}{2})\psi_{sj}(-\frac{1}{2})\mathbbm{1}$, $1\leqslant r, s \leqslant m$, $1\leqslant i,j \leqslant 3$. Moreover, the vertex operator subalgebra of $L_{\widehat{\mathfrak{so}_{3m}}}(1,0)$ generated by $\sum_{r=1}^{m}\psi_{ri}(-\frac{1}{2})\psi_{rj}(-\frac{1}{2})\mathbbm{1}$, $1\leqslant i,j \leqslant 3$, is isomorphic to $L_{\widehat{\mathfrak{sl}_{2}}}(2m,0)$. We simply denote this vertex operator subalgebra by $L_{\widehat{\mathfrak{sl}_{2}}}(2m,0)$, and this gives a fermionic construction of the simple vertex operator algebra $L_{\widehat{\mathfrak{sl}_2}}(2m,0)$.

Following \cite{JLam19}, for  $m \geqslant 4$,  we define the vertex operator algebra automorphism $\sigma_i'$ of $L_{\widehat{\mathfrak{so}_{3m}}}(1,0)$ for $1 \leqslant i \leqslant 3$ as follows
\[ \sigma_i'(\psi_{jl}(-\frac{1}{2})\psi_{rs}(-\frac{1}{2})\mathbbm{1})=(-1)^{\delta_{il}+\delta_{is}}\psi_{jl}(-\frac{1}{2})\psi_{rs}(-\frac{1}{2})\mathbbm{1}. \]
Then
\[ (\sigma_i')^2=id. \]

It is easy to check that the automorphism group of $L_{\widehat{\mathfrak{so}_{3m}}}(1,0)$ generated by $\{\sigma_i', 1\leqslant i\leqslant 3\}$ is  isomorphic to the  Klein group ${K}=\mathbb{Z}_2 \times \mathbb{Z}_2$.
Set
\[L_{\widehat{\mathfrak{so}_{3m}}}(1,0)^{K} = \{ v\in L_{\widehat{\mathfrak{so}_{3m}}}(1,0) \} | gv = v, g \in K \}.\]

We have the following result from \cite{JLam19}.
\begin{thm}\label{jl1}
For $m\geqslant 4$,
$$C_{L_{\wo}(1,0)^{\otimes 3}}(L_{\wo}(3,0))=(L_{\widehat{\frak{sl}_2}}(2m,0)+L_{\widehat{\frak{sl}_2}}(2m,2m))^{K}$$
if $m$ is even, and
$$C_{L_{\wo}(1,0)^{\otimes 3}}(L_{\wo}(3,0))=L_{\widehat{\frak{sl}_2}}(2m,0)^{K}$$
if $m$ is odd.
\end{thm}
Let
\[w^1=\sum_{r=1}^{m}\psi_{r2}(-\frac{1}{2})\psi_{r3}(-\frac{1}{2})\mathbbm{1},\]
\[w^2=\sum_{r=1}^{m}\psi_{r1}(-\frac{1}{2})\psi_{r3}(-\frac{1}{2})\mathbbm{1},\]
\[w^3=\sum_{r=1}^{m}\psi_{r1}(-\frac{1}{2})\psi_{r2}(-\frac{1}{2})\mathbbm{1}.\]
Using the non-trivial relations $[\psi_{ij}(r), \psi_{tl}(s)]_+ = \delta_{it}\delta_{jl}\delta_{r+s,0}$ and the Borcherds identities \eqref{Borcherds identity 1} and \eqref{Borcherds identity 2} , we obtain the following relations:
\[[w^1, w^2] = w^3, \quad [w^2, w^3] = w^1, \quad [w^3, w^1] = w^2.\]
Let
$$h=2\sqrt{-1}w^1 ,    \  e=\sqrt{-1}w^2-w^3 ,  \  f=\sqrt{-1}w^2+w^3.$$
Then $\{e,f,h\}$ are the generators of $L_{\widehat{\mathfrak{sl}_2}}(2m,0)$ and  form a standard basis of $L(0)$-eigenspace $(L_{\widehat{\mathfrak{sl}_2}}(2m,0))_1\cong \frak{sl}_2(\C)$. It is easy to check that on $L_{\widehat{\mathfrak{sl}_2}}(2m,0)$,
$$
\sigma_1=\sigma_1', \ \sigma_2=\sigma_2'.
$$

\section{Classification and construction of irreducible modules of $L_{\widehat{\mathfrak{sl}_2}}(k,0)^{K}$ and $C_{{L_{\widehat{\mathfrak{so}_m}}(1,0)}^{\otimes 3}}({L_{\widehat{\mathfrak{so}_m}}(3,0)})$}

In this section, we will classify and construct explicitly the irreducible modules of the orbifold vertex operator algebras $L_{\widehat{\mathfrak{sl}_2}}(k,0)^{K}$ for $k\geqslant 1$ and the commutant vertex operator algebras $C_{{L_{\widehat{\mathfrak{so}_m}}(1,0)}^{\otimes 3}}({L_{\widehat{\mathfrak{so}_m}}(3,0)})$ for $m\geqslant 4$.

Let Irr$(K)$ denote the set of irreducible characters of ${K}$ which only contains four irreducible characters $\chi_0$ (unit representation), $\chi_1$, $\chi_2$ and $\chi_3$ up to isomorphism with the following character table.
\[
\begin{tabular}{ccccc}
\hline
           &   $1$   &  $\sigma_1$  &  $\sigma_2$  &  $\sigma_3$  \\
\hline
$\chi_0$     &   $1$   &     $1$      &      $1$     &      $1$     \\
$\chi_1$     &   $1$   &     $1$      &     $-1$     &     $-1$     \\
$\chi_2$     &   $1$   &    $-1$      &      $1$     &     $-1$     \\
$\chi_3$     &   $1$   &    $-1$      &     $-1$     &      $1$     \\
\hline
\end{tabular}
\]

\vskip0.3cm
For simplicity, we denote $L_{\widehat{\mathfrak{sl}_2}}(k,0)$ by $L(k,0)$. We first have the following decomposition.

\begin{thm}   \label{V decomposition}
As a $L(k,0)^{K}$-module,
\[ L(k,0)= \oplus_{j=0}^{3}L(k,0)^{(j)},\]
where $L(k,0)^{(0)} (= L(k,0)^{K}) ($resp. $L(k,0)^{(1)}, L(k,0)^{(2)}, L(k,0)^{(3)} )$ is the irreducible $L(k,0)^{K}$-module generated by the lowest weight vector $\mathbbm{1} ($resp. $h$, $e+f$, $e-f )$ with the lowest weight $0 ($resp. $1, 1, 1 )$.
\end{thm}
\begin{proof}
Since ${K}$ is abelian, all the irreducible modules of $\mathbb{C}K$ are one dimensional. From \cite{DM97}, ${K} = \oplus_{\chi \in \Irr(K)}K_{\chi}$ is a decomposition of $L(k,0)$ into simple $L(k,0)^{K}$-modules. Moreover, ${K}_{\chi}$ is nonzero for any $\chi \in \Irr(K)$, and $L(k,0)_{\chi}$ and $L(k,0)_{\mu}$ are equivalent $L(k,0)^{K}$-module if and only if $\chi=\mu$. Obviously, $L(k,0)^{K}$ is an irreducible $L(k,0)^{K}$-module affording the unit character $\chi_0$. Observing the action of ${K}$ on $L(k,0)_1$ which is isomorphic to $\mathfrak{sl}_2(\C)$, we find that $h$, $e+f$, $e-f$ generate three inequivalent irreducible modules according to $\chi_1$, $\chi_2$, $\chi_3$, respectively. Note that $L(k,0)_0 = \mathbb{C}\mathbbm{1}$ and $L(k,0)_1 = \C h \oplus \mathbb{C}(e+f) \oplus \C(e-f)$, then $\mathbbm{1}$, $h$, $e+f$ and $e-f$ are four different lowest weight vectors in $L(k,0)$ as a $L(k,0)^{K}$-module. Let $L(k,0)^{(0)} ($resp. $L(k,0)^{(1)}, L(k,0)^{(2)}, L(k,0)^{(3)} )$ be the irreducible $L(k,0)^{K}$-module generated by the lowest weight vector $\mathbbm{1} ($resp. $h$, $e+f$, $e-f )$ with the lowest weight $0 ($resp. $1, 1, 1 )$. Then the irreducible $L(k,0)^{K}$-module decomposition $L(k,0)= \oplus_{j=0}^{3}L(k,0)^{(j)}$ holds.
\end{proof}

Let $\alpha$ be the simple root of $\mathfrak{sl}_2(\C)$ with $\langle \alpha, \alpha  \rangle = 2$. From \cite{FZ92}, the integrable highest weight $L(k,0)$-modules $L(k, i)$ for $0 \leqslant i \leqslant k$ provide a complete list of irreducible $L(k,0)$-modules with the lowest weight spaces being  $(i+1)$-dimensional irreducible $\mathfrak{sl}_2(\C)$-modules $L(\frac{i\alpha}{2})$, respectively.
For  $r = 1$, $2$, $3$, $0\leqslant i\leqslant k$, let $\{h^{(r)}, e^{(r)}, f^{(r)}\}$ be the same as introduced in Section 3, and $\{v^{r,i,j}|0 \leqslant j \leqslant i\}$  the basis of $L(\frac{i\alpha}{2})$ according to the $\mathfrak{sl}_2$-triple $\{ h^{(r)}, e^{(r)}, f^{(r)} \}$ with the following action of $\widehat{\mathfrak{sl}_2}$ on $L(\frac{i\alpha}{2})$, namely
\[ h^{(r)}(0)v^{r,i,j} = (i - 2j) v^{r,i,j} \quad \text{for} \quad 0 \leqslant j \leqslant i , \]
\[ e^{(r)}(0)v^{r,i,0} = 0, \quad e^{(r)}(0)v^{r,i,j} = (i - j + 1)v^{r,i,j-1} \quad \text{for} \quad 1 \leqslant j \leqslant i , \]
\[ f^{(r)}(0)v^{r,i,i} = 0, \quad f^{(r)}(0)v^{r,i,j} = (j + 1)v^{r,i,j+1} \quad \text{for} \quad 0 \leqslant j \leqslant i-1 , \]
\[ a^{(r)}(n)v^{r,i,j} = 0 \quad \text{for} \quad a \in \{ h, e, f \}, \quad n \geqslant 1 .\]

Recall that $h^{(r)'}=\frac{1}{4}h^{(r)}$. For  $r = 1$, $2$, $3$, let
\[\Delta(h^{(r)'}, z) = z^{h^{(r)'}(0)}\exp(\sum^{\infty}_{n=1}\frac{h^{(r)'}(n)}{-n}(-z)^{-n}).\]
From \cite{Li97-2}, we have the following result.
\begin{lem}
For each $r = 1$, $2$, $3$, $( \overline{L(k,i)}^{\sigma_r}, Y_{\sigma_{r}}(\cdot, z)) = ( L(k,i), Y(\Delta(h^{(r)'}, z)\cdot, z) ) ( 0 \leqslant i \leqslant k )$ provide a complete list of irreducible $\sigma_{r}$-twisted $L(k, 0)$-modules.
\end{lem}

Direct calculations (also see \cite{JW1}) yield that for any $r = 1$, $2$, $3$,
\begin{equation}
h^{(r)'}(0)\omega = 0, \quad h^{(r)'}(1)\omega = h^{(r)'}, \quad h^{(r)'}(1)^{2}\omega = \frac{k}{8}\mathbbm{1},            \label{twist-h1}
\end{equation}
\begin{equation}
h^{(r)'}(2)\omega = 0, \quad h^{(r)'}(n)\omega = 0  \quad \text{for} \quad n \in \mathbb{Z}_{>2},                          \label{twist-h2}
\end{equation}
\begin{equation}
\Delta(h^{(r)'}, z)\omega = \omega + z^{-1}h^{(r)'} + z^{-2}\frac{k}{16}\mathbbm{1},  \label{twist-h3}
\end{equation}
\begin{equation}
Y_{\sigma_{r}}(h^{(r)'}, z) = Y(h^{(r)'}+\frac{k}{8}z^{-1}, z),   \label{twist-h4}
\end{equation}
\begin{equation}
Y_{\sigma_{r}}(h^{(r)}, z) = Y(h^{(r)}+\frac{k}{2}z^{-1}, z),    \label{twist-h5}
\end{equation}
\begin{equation}
Y_{\sigma_{r}}(e^{(r)}, z) =z^{\frac{1}{2}}Y(e^{(r)}, z),   \label{twist-e}
\end{equation}
\begin{equation}
Y_{\sigma_{r}}(f^{(r)}, z) =z^{-\frac{1}{2}}Y(f^{(r)}, z).     \label{twist-f}
\end{equation}
To distinguish the components of $Y(v,z)$ from those of $Y_{\sigma_{r}}(v,z) ( r = 1, 2, 3 )$, we denote the following expansions
\[ Y_{\sigma_{r}}(v,z) = \sum_{n\in \frac{t}{2}+\mathbb{Z}}v_nz^{-n-1}, \quad  Y(v,z) = \sum_{n\in\mathbb{Z}}v(n)z^{-n-1}, \]
where $v \in L(k, 0)$, $t \in \{0, 1\}$ such that $\sigma_{r}(v) = e^{-\pi \sqrt{-1}t}v$ for some $r \in \{1, 2, 3\}$.

For $0 \leqslant j \leqslant i$, we denote $v^{1,i,j}$ by $v^{i,j}$ for convenience.
For each $r = 1$, $2$, $3$, since $\{v^{r,i,j}|0 \leqslant j \leqslant i\}$ is a specific basis of $L(\frac{i\alpha}{2})$ according to the $\mathfrak{sl}_2$-triple $\{ h^{(r)}, e^{(r)}, f^{(r)} \}$, we can use $v^{i,j} ( 0 \leqslant j \leqslant i )$ to express the specific form of $v^{r,i,i}$ as follows:
\[ v^{1,i,i} = v^{i,i}, \quad v^{2,i,i} = \sum_{j=0}^{i}(-1)^jv^{i,j}, \quad v^{3,i,i} = \sum_{j=0}^{i}(\sqrt{-1})^jv^{i,j}. \]
Keeping in mind that the conformal vector $\omega$ belongs to $L(k,0)^{K}$, we can obtain the following lemma by a straightforward calculation (see also \cite{JW1} and \cite{JW2}).

\begin{lem} \label{twisted l.w.}
For $0 \leqslant i \leqslant k$, $r = 1$, $2$, $3$, we have $L_0v^{r,i,i} = (\frac{i(i-k)}{4(k+2)} + \frac{k}{16})v^{r,i,i}$.
\end{lem}

By induction, we can prove the following lemma.
\begin{lem}
For each $r = 1$, $2$, $3$, $\{ h^{(r)}(-1)\mathbbm{1}, e^{(r)}(-1)^2\mathbbm{1}, f^{(r)}(-1)^2\mathbbm{1}\}$ is a generator set of the orbifold vertex operator subalgebra $L(k, 0)^{\langle \sigma_{r} \rangle}$ of $L(k, 0)$.
\end{lem}

Now we are poised to give the classification of the irreducible modules for the orbifold vertex operator subalgebra $L(k,0)^{\langle \sigma_{r} \rangle}$ of $L(k,0)$ for each $r = 1$, $2$, $3$. The same result for $\sigma_2$ was given  in the literature \cite{JW2}, here we include each $\sigma_r,  r = 1, 2, 3$ for completeness.

\begin{thm}   \label{Z_2 decomposition}
For $r = 1$, $2$, or $3$, there are $4(k+1)$ irreducible $L(k, 0)^{\langle \sigma_r \rangle}$-modules up to isomorphism as follows:
$$
L(k,i)^{\sigma_r,+},  \ L(k,i)^{\sigma_r,-}, \ \overline{L(k,i)}^{\sigma_r,+}, \  \overline{L(k,i)}^{\sigma_r,-}, \ 0\leqslant i\leqslant k.
$$
\end{thm}
\begin{proof}
For $i = 0$, $r = 1$, $2$, $3$, the decomposition of $L(k,0)$ into inequivalent irreducible $L(k,0)^{\langle \sigma_{r} \rangle}$-modules is
\begin{equation}   \label{Z_2 decomposition eq0}
L(k,0) = L(k,0)^{\sigma_r,+} \oplus L(k,0)^{\sigma_r,-},
\end{equation}
where $L(k,0)^{\sigma_r,+}$ is generated by the lowest weight vector $\mathbbm{1}$ with the lowest weight $0$ as an irreducible $L(k, 0)^{\langle \sigma_{r} \rangle}$-module, and $L(k,0)^{\sigma_r,-}$ is generated by the lowest weight vector $e^{(r)}(-1)\mathbbm{1}$ with the lowest weight $1$ as an irreducible $L(k,0)^{\langle \sigma_{r} \rangle}$-module.
For $0 < i \leqslant k$, the decomposition of $L(k,i)$ into inequivalent irreducible $L(k,0)^{\langle \sigma_r \rangle}$-modules is
\begin{equation}   \label{Z_2 decomposition eq1}
L(k,i) = L(k,i)^{\sigma_r,+} \oplus L(k,i)^{\sigma_r,-},
\end{equation}
where $L(k,i)^{\sigma_r,+}$ is generated by the lowest weight vector $v^{r,i,i}$ with the lowest weight $\frac{i(i+2)}{4(k+2)}$ as an irreducible $L(k, 0)^{\langle \sigma_{r} \rangle}$-module, and $L(k,i)^{\sigma_r,-}$ is generated by the lowest weight vector $v^{r,i,i-1}$ with the lowest weight $\frac{i(i+2)}{4(k+2)}$ as an irreducible $L(k, 0)^{\langle \sigma_{r} \rangle}$-module.

From \cite{DM97}, we know that if $V$ is a vertex operator algebra with an automorphism $g$ of order $T$, and $M=\oplus_{n\in \frac{1}{T}\mathbb{Z}_+}M(n)$ is an irreducible $g$-twisted admissible module of $V$, then $M^t=\oplus_{n\in \frac{t}{T}+\mathbb{Z}_+}M(n)$ is an irreducible $V^{\langle g \rangle}$-module for $t=0, \cdots, T-1$.
Therefore, the decomposition of $\overline{L(k,i)}^{\sigma_r} (0 \leqslant i \leqslant k)$ into inequivalent irreducible $L(k,0)^{\langle \sigma_{r} \rangle}$-modules is
\begin{equation}    \label{Z_2 decomposition eq2}
\overline{L(k,i)}^{\sigma_r} = \overline{L(k,i)}^{\sigma_r,+} \oplus \overline{L(k,i)}^{\sigma_r,-},
\end{equation}
where $\overline{L(k,i)}^{\sigma_r,+}$ is generated by the lowest weight vector $v^{r,i,i}$ with the lowest weight $\frac{i(i-k)}{4(k+2)} + \frac{k}{16}$ as an irreducible $L(k, 0)^{\langle \sigma_{r} \rangle}$-module, and $\overline{L(k,i)}^{\sigma_r,-}$ is generated by the lowest weight vector $f^{(r)}_{-\frac{1}{2}}v^{r,i,i}$ with the lowest weight $\frac{i(i-k)}{4(k+2)} + \frac{k}{16}+\frac{1}{2}$ as an irreducible $L(k,0)^{\langle \sigma_{r} \rangle}$-module. Thus there are $4(k+1)$ inequivalent irreducible $L(k,0)^{\langle \sigma_r \rangle}$-modules for each $r = 1$, $2$, $3$.
\end{proof}

\vskip 0.3cm
For $0\leqslant i\leqslant k, \ 0\leqslant j\leqslant k,\ 0\leqslant l\leqslant k$ such that $i+j+l\in 2\mathbb{Z}$, following \cite{JW2}, we define

\[\mbox{sign}(i,j,l)^{+}=\begin{cases}+,\ & \mbox{if} \ i+j-l\in 4{\mathbb{Z}},\cr
-,\ &\mbox{if} \  i+j-l\notin 4{\mathbb{Z}},\end{cases}\]

and

\[\mbox{sign}(i,j,l)^{-}=\begin{cases}-,\ & \mbox{if} \ i+j-l\in 4{\mathbb{Z}},\cr
+,\ &\mbox{if} \  i+j-l\notin 4{\mathbb{Z}}.\end{cases}\]

For $r=1,2,$ or $3$, we have the following fusion rules for the ${\Z}_{2}$-orbifold affine vertex operator algebra $L(k,0)^{\sigma_r}$.

\begin{thm}\label{fusion-aff} The fusion rules for the ${\mathbb{Z}}_{2}$-orbifold affine vertex operator algebra $L(k,0)^{\la \sigma_r \ra}$ are as follows:
\begin{eqnarray}\label{fusion.untwist1.}
L(k,i)^{\sigma_r,+}\boxtimes L(k,j)^{\sigma_r,\pm}=\sum\limits_{\tiny{\begin{split}|i-j|\leqslant l\leqslant i+j \\  i+j+l\in 2\mathbb{Z} \ \ \ \\ i+j+l\leqslant 2k\ \ \ \end{split}}} L(k,l)^{\sigma_r, \mbox{sign}(i,j,l)^{\pm}},
\end{eqnarray}

\begin{eqnarray}\label{fusion.untwist2}
L(k,i)^{\sigma_r,-}\boxtimes L(k,j)^{\sigma_r,\pm}=\sum\limits_{\tiny{\begin{split}|i-j|\leqslant l\leqslant i+j \\  i+j+l\in 2\mathbb{Z} \ \ \ \\ i+j+l\leqslant 2k\ \ \ \end{split}}}  L(k,l)^{\sigma_r, \mbox{sign}(i,j,l)^{\mp}},
\end{eqnarray}

\begin{eqnarray}\label{fusion.twist1}
L(k,i)^{\sigma_r,+}\boxtimes \overline{L(k,j)}^{\sigma_r,\pm}=\sum\limits_{\tiny{\begin{split}|i-j|\leqslant l\leqslant i+j \\  i+j+l\in 2\mathbb{Z} \ \ \ \\ i+j+l\leqslant 2k\ \ \ \end{split}}} \overline{L(k,l)}^{\sigma_r, \mbox{sign}(i,j,l)^{\pm}},
\end{eqnarray}

\begin{eqnarray}\label{fusion.twist2}
L(k,i)^{\sigma_r,-}\boxtimes \overline{L(k,j)}^{\sigma_r,\pm}=\sum\limits_{\tiny{\begin{split}|i-j|\leqslant l\leqslant i+j \\  i+j+l\in 2\mathbb{Z} \ \ \ \\ i+j+l\leqslant 2k\ \ \ \end{split}}}  \overline{L(k,l)}^{\sigma_r, \mbox{sign}(i,j,l)^{\mp}}.
\end{eqnarray}
\end{thm}

Now we are in a position to classify and construct all the irreducible modules for the orbifold vertex operator algebra $L(k,0)^{K}$. Since $L(k,i)(0 \leqslant i \leqslant k)$ are all the inequivalent irreducible $L(k,0)$-modules which can be viewed as id-twisted $L(k,0)$-modules, we first determine the subgroup ${K}_{L(k,i)}$ of ${K}$ which contains   $\sigma \in K$ such that $L(k,i)$ is $\sigma$-stable. From now on, we denote id by $\sigma_0$ for convenience.

\begin{lem}  \label{irr K_M subgroup}
${K}_{L(k,i)} = K$ for any $0 \leqslant i \leqslant k$.
\end{lem}
\begin{proof}

By definition, $L(k,i)$ and $L(k,i) \circ \sigma_r$ have the same lowest weight. Observe that the lowest weights $\frac{i(i+2)}{4(k+2)}(0 \leqslant i \leqslant k)$ are pairwise different which implies that all the irreducible $L(k, 0)$-modules $L(k,i)(0 \leqslant i \leqslant k)$ are $\sigma_r$-stable for any $r = 0$, $1$, $2$, $3$. Thus, ${K}_{L(k,i)} = K$.
\end{proof}

Define $\phi(\sigma_r) ( r = 0, 1, 2, 3 )$ as follows:
\begin{align}
 \phi(\sigma_0): v^{i,j} & \mapsto v^{i,j},       \label{phi 0} \\
 \phi(\sigma_1): v^{i,j} & \mapsto (-1)^jv^{i,j},  \label{phi 1}\\
 \phi(\sigma_2): v^{i,j} & \mapsto v^{i,i-j},      \label{phi 2}\\
 \phi(\sigma_3): v^{i,j} & \mapsto (-1)^{i-j}v^{i,i-j}. \label{phi 3}
\end{align}
It is easy to verify that $\phi(\sigma_r)$ for $r = 0$, $1$, $2$, $3$ are $L(k,0)$-module isomorphisms.

\begin{lem} \label{irr decomposition}
For each $0 < i \leqslant k$, we have the following irreducible $L(k,0)^{K}$-module decomposition.
\begin{enumerate}
\item
If $i \in 2\mathbb{Z}+1$, then
\begin{equation}
L(k,i) = W(i) \otimes L(k,i)^{\sigma_1,+} = L(k,i)^{\sigma_1,+} \oplus L(k,i)^{\sigma_1,-},
\end{equation}
where $W(i)$ is the unique two-dimensional irreducible module of the twisted group algebra $\mathbb{C}^{\alpha_{L(k,i)}}[K]$. Moreover, $L(k,i)^{\sigma_1,+}$ and $L(k,i)^{\sigma_1,-}$ are isomorphic irreducible $L(k,0)^{K}$-modules.
\item
If $i = 2$, then
\begin{equation}
L(k,2) = \bigoplus_{j=0}^{j=3}L(k,2)^{(j)},
\end{equation}
where $L(k,2)^{(0)}$, $L(k,2)^{(1)}, L(k,2)^{(2)}$ and $L(k,2)^{(3)}$ are the  irreducible $L(k,0)^{K}$-modules generated by the lowest weight vectors $v^{2,0}+v^{2,2}$, $v^{2,0}-v^{2,2}$, $v^{2,1}$, and $h(-1)v^{2,1}$ with the lowest weights $\frac{2}{k+2}$,  $\frac{2}{k+2}$, $\frac{2}{k+2}$, and $\frac{k+4}{k+2}$, respectively.
\item
If $i \in 2\mathbb{Z}_{>1}$, then
\begin{equation}
L(k,i) = \bigoplus_{j=0}^{j=3}L(k,i)^{(j)},
\end{equation}
where $L(k,i)^{(0)}$, $L(k,i)^{(1)}, L(k,i)^{(2)}$, and $L(k,i)^{(3)}$ are the irreducible $L(k,0)^{K}$-modules generated by the lowest weight vectors $v^{i,0} + v^{i,i}$,  $v^{i,0} - v^{i,i}$, $v^{i,1} + v^{i,i-1}$, and $v^{i,1} - v^{i,i-1}$ with the same lowest weight $\frac{i(i+2)}{4(k+2)}$, respectively.
\end{enumerate}
\end{lem}
\begin{proof}
 By Lemma \ref{irr K_M subgroup}, ${K}_{L(k,i)} = K$ for  $0 \leqslant i \leqslant k$, and   the ${K}
 $-orbit $L(k,i) \circ K$ of $L(k,i)$ only contains itself for any $0 \leqslant i \leqslant k$. The simplicity of $L(k,i)$ together with Schur's lemma shows that $\sigma_r \mapsto \phi(\sigma_r) ( r = 0, 1, 2, 3 )$ gives a projective representation of ${K}$ on $L(k,i)$. Let $\alpha_{L(k,i)}$ be the corresponding $2$-cocycle in $C^2(K, \mathbb{C}^*)$. Then $L(k,i)$ is a module for the twisted group algebra $\mathbb{C}^{\alpha_{L(k,i)}}[K]$ with relations
\[\phi(\sigma_1)\phi(\sigma_2) = \phi(\sigma_1\sigma_2) = \phi(\sigma_3), \quad \phi(\sigma_2)\phi(\sigma_1) = (-1)^i\phi(\sigma_2\sigma_1) = (-1)^i\phi(\sigma_3).  \]

If $i \in 2\mathbb{Z}+1$, the twisted group algebra $\mathbb{C}^{\alpha_{L(k,i)}}[K]$ is a non-commutative semisimple associative algebra which has only one irreducible module of dimension two, and we denote it by $W(i)$. Recall the decomposition \eqref{Z_2 decomposition eq1}, we can deduce that $L(k,i)^{\sigma_1,+}$ and $L(k,i)^{\sigma_1,-}$ are irreducible $L(k,0)^{K}$-modules. Then $L(k,i)^{\sigma_1,+}$ is the multiplicity of $W(i)$ in $L(k,i)$, and the decomposition $L(k,i) = W(i) \otimes L(k,i)^{\sigma_1,+} = L(k,i)^{\sigma_1,+} \oplus L(k,i)^{\sigma_1,-}$ holds. Then $L(k,i)^{\sigma_1,+}$ and $L(k,i)^{\sigma_1,-}$ are isomorphic irreducible $L(k,0)^{K}$-modules due to Theorem \ref{DRX17 thm1}.

If $i \in 2\mathbb{Z}$, the twisted group algebra $\mathbb{C}^{\alpha_{L(k,i)}}[K]$ is a commutative semisimple associative algebra which has four irreducible modules of dimension one.
For $i = 0$, the decomposition of $L(k,0)$ into irreducible $L(k,0)^{K}$-modules is exactly the case in Theorem \ref{V decomposition}.
For $i = 2$, let $L(k,2)^{(0)} ($resp. $L(k,2)^{(1)}, L(k,2)^{(2)}, L(k,2)^{(3)} )$ be the irreducible $L(k,0)^K$-module generated by the lowest weight vector $v^{2,0}+v^{2,2} ($resp. $v^{2,0}-v^{2,2}$, $v^{2,1}$, $h(-1)v^{2,1} )$ with the lowest weight $\frac{2}{k+2} ($resp. $\frac{2}{k+2}, \frac{2}{k+2}, \frac{k+4}{k+2} )$. Then $L(k,2)^{\sigma_1,+} = L(k,2)^{(0)} \oplus L(k,2)^{(1)}$ and $L(k,2)^{\sigma_1,-} = L(k,2)^{(2)} \oplus L(k,2)^{(3)}$. Thus $L(k,2) = \bigoplus_{j=0}^{j=3}L(k,2)^{(j)}$ is a decomposition of $L(k,2)$ into inequivalent irreducible $L(k,0)^{K}$-modules.
For $i \in 2\mathbb{Z}_{>1}$, let $L(k,i)^{(0)}$, $L(k,i)^{(1)}$, $L(k,i)^{(2)}$, and $L(k,i)^{(3)}$ be the irreducible $L(k,0)^{K}$-modules generated by the lowest weight vectors $v^{i,0} + v^{i,i}$, $v^{i,0} - v^{i,i}$, $v^{i,1} + v^{i,i-1}$, and $v^{i,1} - v^{i,i-1}$ with the same lowest weight $\frac{i(i+2)}{4(k+2)}$, respectively. Then $L(k,i)^{\sigma_1,+} = L(k,i)^{(0)} \oplus L(k,i)^{(1)}$ and $L(k,i)^{\sigma_1,-} = L(k,i)^{(2)} \oplus L(k,i)^{(3)}$.
From Theorem \ref{DRX17 thm1},
we know that $L(k,i)^{(j)}, j = 0, 1, 2, 3$ are inequivalent irreducible $L(k,0)^{K}$-modules for fixed $2\leqslant i\leqslant k$, $i\in 2\Z$.
Thus the decomposition of $L(k,i)$ into inequivalent irreducible $L(k,0)^{K}$-modules is
$L(k,i) = \oplus_{j=0}^{j=3}L(k,i)^{(j)}$.
\end{proof}

\begin{rmk}
If $0 \leqslant i \leqslant k$, $i \in 2\mathbb{Z}+1$, the set $\{ L(k,i)^{\sigma_r,\pm}|r = 1, 2, 3 \}$ is an equivalence class of irreducible $L(k,0)^{K}$-modules.
From now on, we denote any $L(k,0)^{K}$-module in $\{ L(k,i)^{\sigma_r,\pm}|r = 1, 2, 3 \}$ by $L(k,i)^+$ for the sake of simplicity.
\end{rmk}

\begin{rmk}
If $4 \leqslant i \leqslant k$, $i \in 2\mathbb{Z}$, the top levels of the irreducible $L(k,0)^{K}$-modules $L(k,i)^{(j)} (j = 0, 1, 2, 3 )$ are listed below.\\
If $i \in 4\mathbb{Z}$, the top level of $L(k,i)^{(0)}$ is a vector space with a bisis
\[ \{ v^{i,0} + v^{i,i}, v^{i,2} + v^{i,i-2}, \cdots , v^{i,\frac{i}{2}} \} ; \]
the top level of $L(k,i)^{(1)}$ is a vector space with a bisis
\[ \{ v^{i,0} - v^{i,i}, v^{i,2} - v^{i,i-2}, \cdots , v^{i,\frac{i}{2}-2}-v^{i,\frac{i}{2}+2} \} ; \]
the top level of $L(k,i)^{(2)}$ is a vector space with a bisis
\[ \{ v^{i,1}+v^{i,i-1}, v^{i,3}+v^{i,i-3}, \cdots , v^{i,\frac{i}{2}-1}+v^{i,\frac{i}{2}+1} \} ; \]
the top level of $L(k,i)^{(3)}$ is a vector space with a bisis
\[ \{ v^{i,1}-v^{i,i-1}, v^{i,3}-v^{i,i-3}, \cdots , v^{i,\frac{i}{2}-1}-v^{i,\frac{i}{2}+1} \} . \]
If $i \in 4\mathbb{Z}+2$, the top level of $L(k,i)^{(0)}$ is a vector space with a basis
\[ \{ v^{i,0} + v^{i,i}, v^{i,2} + v^{i,i-2}, \cdots , v^{i,\frac{i}{2}-1} + v^{i,\frac{i}{2}+1} \}  ; \]
the top level of $L(k,i)^{(1)}$ is a vector space with a bisis
\[ \{ v^{i,0} - v^{i,i}, v^{i,2} - v^{i,i-2}, \cdots , v^{i,\frac{i}{2}-1} - v^{i,\frac{i}{2}+1} \}  ; \]
the top level of $L(k,i)^{(2)}$ is a vector space with a bisis
\[ \{ v^{i,1}+v^{i,i-1}, v^{i,3}+v^{i,i-3}, \cdots , v^{i,\frac{i}{2}} \} ; \]
the top level of $L(k,i)^{(3)}$ is a vector space with a bisis
\[ \{ v^{i,1}-v^{i,i-1}, v^{i,3}-v^{i,i-3}, \cdots , v^{i,\frac{i}{2}-2}-v^{i,\frac{i}{2}+2} \} . \]
\end{rmk}

In conclusion, if $k$ is odd, there are $\frac{k+1}{2}$ inequivalent irreducible $L(k,0)^{K}$-modules coming from $L(k,i), 0 \leqslant i\leqslant k, i \in 2\mathbb{Z}+1$, and  $2(k+1)$ inequivalent irreducible $L(k,0)^{K}$-modules coming from $L(k,i), 0 \leqslant i\leqslant k, i \in 2\mathbb{Z}$. Therefore, there are $\frac{5(k+1)}{2}$ inequivalent irreducible $L(k,0)^{K}$-modules coming from $L(k,i), 0 \leqslant i\leqslant k$. If $k$ is even, there are $\frac{k}{2}$ inequivalent irreducible $L(k,0)^{K}$-modules coming from $L(k,i), 0 \leqslant i\leqslant k, i \in 2\mathbb{Z}+1$, and $2(k+2)$ inequivalent irreducible $L(k,0)^{K}$-modules coming from $L(k,i), 0 \leqslant i\leqslant k, i \in 2\mathbb{Z}$. Therefore, there are  $\frac{5k+8}{2}$ inequivalent irreducible $L(k,0)^{K}$-modules coming from $L(k,i), 0 \leqslant i\leqslant k$.

\begin{lem}   \label{twisted G_M subgroup}
For $0 \leqslant i \leqslant k$, and $r = 1$, $2$, $3$, we have
$$
\overline{L(k,i)}^{\sigma_r, +}\cong \overline{L(k,k-i)}^{\sigma_r,+}, \ \overline{L(k,i)}^{\sigma_r, -}\cong \overline{L(k,k-i)}^{\sigma_r,-}
$$
as $L(k,0)^{K}$-modules, and
\[  K_{\overline{L(k,i)}^{\sigma_r}} =
\begin{cases}
\langle \sigma_{r} \rangle,  & \mbox{if }  i \neq \frac{k}{2} ,  \\
K, & \mbox{if }   i = \frac{k}{2}, \ k\in 2\Z_{+}.
\end{cases}  \]
\end{lem}
\begin{proof}
We will prove the case for $r = 1$, the proof of the case for $r = 2, 3$ is similar. Note that $\overline{L(k,i)}^{\sigma_1} ( 0 \leqslant i \leqslant k )$ are all the irreducible $\sigma_1$-twisted $L(k,0)$-modules up to isomorphism.
Similar to \eqref{phi 0} and \eqref{phi 1}, we define $\phi(\sigma_0)$ maps $v^{i,i}$ to $v^{i,i}$ and $\phi(\sigma_1)$ maps $v^{i,i}$ to $(-1)^iv^{i,i}$ where we use the same notation $\phi$. Repeated application of \eqref{twist-h1}-\eqref{twist-f} along with the following relations
\[ e_n = e(n + \frac{1}{2}) , \quad f_n = f(n - \frac{1}{2}) , \quad n \in \frac{1}{2} + \mathbb{Z} , \quad h_n =
\begin{cases}
h(n),  & \mbox{if }  n \neq 0 ,  \\
h(0)+\frac{k}{2}, & \mbox{if } n = 0 ,
\end{cases}    \]
we can see that $\phi(\sigma_0)$ and $\phi(\sigma_1)$ are $\sigma_1$-twisted $L(k,0)$-module isomorphisms. Then we have the following $\sigma_{1}$-twisted $L(k,0)$-module isomorphisms
\[ \overline{L(k,i)}^{\sigma_1} \circ \sigma_0 \cong \overline{L(k,i)}^{\sigma_1}, \quad \overline{L(k,i)}^{\sigma_1} \circ \sigma_1 \cong \overline{L(k,i)}^{\sigma_1}, \]
\[ \overline{L(k,i)}^{\sigma_1} \circ \sigma_2  \cong (\overline{L(k,i)}^{\sigma_1} \circ \sigma_1) \circ \sigma_3 \cong \overline{L(k,i)}^{\sigma_1} \circ \sigma_3. \]
In order to prove ${K}_{\overline{L(k,i)}^{\sigma_1}} = \{ \sigma_0, \sigma_1 \}$ for $i \neq \frac{k}{2}$, we need to confirm that $\overline{L(k,i)}^{\sigma_1} ( i \neq \frac{k}{2} )$ are neither $\sigma_2$-stable nor $\sigma_3$-stable.
Otherwise, $\sigma_2$ belongs to ${K}_{\overline{L(k,i)}^{\sigma_1}}$.
Recall from Theorem \ref{Z_2 decomposition}, $\overline{L(k,i)}^{\sigma_1,+}$ is generated by the lowest weight vector $v^{i,i}$ as an irreducible $L(k,0)^{\langle \sigma_1 \rangle}$-module. Notice that $h(-1)\mathbbm{1} \in L(k,0)^{\langle \sigma_1 \rangle}$ and $h(-1)\mathbbm{1} \notin L(k,0)^{\langle \sigma_2 \rangle}$, then $v^{i,i}$ and $(h(-1)\mathbbm{1})_{-1}v^{i,i} = h(-1)v^{i,i}$ must belong to different direct summands of the decomposition of $\overline{L(k,i)}^{\sigma_1}$ into $L(k,0)^{K}$-modules. On the other hand, $h(-1)^2\mathbbm{1} \in L(k,0)^{K}$ and
\[ (h(-1)^2\mathbbm{1})_{-1}v^{i,i} = 2\sum_{m \geqslant 0}h_{-1-m}h_mv^{i,i} = 2h(-1)(h(0)+\frac{k}{2})v^{i,i} = (k-2i)h(-1)v^{i,i}. \]
As a result of $i \neq \frac{k}{2}$, we see that $h(-1)v^{i,i}$ belongs to the irreducible $L(k,0)^{K}$-module which is generated by $v^{i,i}$, yielding a contradiction. Together with Lemma \ref{twisted l.w.}, we deduce that
\begin{equation}
\overline{L(k,i)}^{\sigma_1} \circ \sigma_2  \cong \overline{L(k,i)}^{\sigma_1} \circ \sigma_3 \cong \overline{L(k,k-i)}^{\sigma_1},
\end{equation}
and
 $${K}_{\overline{L(k,i)}^{\sigma_1}} = \{ \sigma_0, \sigma_1\}$$ for $i \neq \frac{k}{2}$.

If $k \in 2\mathbb{Z}_{+}$, we define $\phi(\sigma_s)$ mapping $v^{\frac{k}{2},\frac{k}{2}}$ to $v^{\frac{k}{2},\frac{k}{2}}$ for $s = 0, 2,$ and $\phi(\sigma_t)$ mapping $v^{\frac{k}{2},\frac{k}{2}}$ to $(-1)^\frac{k}{2}v^{\frac{k}{2},\frac{k}{2}}$ for $t = 1, 3$ where we still use the same notation $\phi$.
It is easy to verify that $\phi(\sigma_j)$ for $j = 0$, $1$, $2$, $3$ are $\sigma_1$-twisted $L(k,0)$-module isomorphisms. Thus ${K}_{\overline{L(k,\frac{k}{2})}^{\sigma_1}} = K$.
\end{proof}

\begin{lem}   \label{lwv for twisted k/2}
If $k \in 2\mathbb{Z}_{+}$, for each $r = 1$, $2$, $3$, the vectors $v^{r,\frac{k}{2},\frac{k}{2}}$, $h^{(r)}(-1)v^{r,\frac{k}{2},\frac{k}{2}}$, $(e^{(r)} + f^{(r)})_{-\frac{1}{2}}v^{r,\frac{k}{2},\frac{k}{2}}$ and $(e^{(r)} - f^{(r)})_{-\frac{1}{2}}v^{r,\frac{k}{2},\frac{k}{2}}$ are four different lowest weight vectors in $\overline{L(k,\frac{k}{2})}^{\sigma_r}$ as $L(k,0)^{K}$-modules, respectively.
\end{lem}
\begin{proof}
We will prove the case for $r = 1$, the proof of the case for $r = 2, 3$ is similar.
We know that $\mathbbm{1}$ (resp. $h(-1)\mathbbm{1}$, $e(-1)\mathbbm{1} + f(-1)\mathbbm{1}$, $e(-1)\mathbbm{1} - f(-1)\mathbbm{1}$) is the generator of $L(k,0)^{(0)}$ (resp. $L(k,0)^{(1)}$, $L(k,0)^{(2)}$, $L(k,0)^{(3)}$).
Since $\phi(\sigma)Y_{\sigma_1}(v, z){\phi(\sigma)}^{-1} = Y_{\sigma_1}(\sigma (v), z)$ for any $v \in L(k,0)$, $\sigma \in K$, we see that
$v^{\frac{k}{2},\frac{k}{2}}$,
$h_{-1}v^{\frac{k}{2},\frac{k}{2}} = h(-1)v^{\frac{k}{2},\frac{k}{2}}$,
$(e + f)_{-\frac{1}{2}}v^{\frac{k}{2},\frac{k}{2}}$ and
$(e - f)_{-\frac{1}{2}}v^{\frac{k}{2},\frac{k}{2}}$
must belong to different direct summands of the decomposition of $\overline{L(k,\frac{k}{2})}^{\sigma_1}$ into $L(k,0)^{K}$-modules.
Because the lowest weight space of $\overline{L(k,\frac{k}{2})}^{\sigma_1,+}$ is a one dimensional vector space spanned by $v^{\frac{k}{2},\frac{k}{2}}$, then $h(-1)v^{\frac{k}{2},\frac{k}{2}}$ is another lowest weight vector in $\overline{L(k,\frac{k}{2})}^{\sigma_1,+}$ as a $L(k,0)^{K}$-module besides $v^{\frac{k}{2},\frac{k}{2}}$. Observe that the top level of $\overline{L(k,\frac{k}{2})}^{\sigma_1,-}$ is a two dimensional vector space spanned by $e_{-\frac{1}{2}}v^{\frac{k}{2},\frac{k}{2}} = v^{\frac{k}{2},\frac{k}{2}-1}$ and $f_{-\frac{1}{2}}v^{\frac{k}{2},\frac{k}{2}}$, and $\{(e + f)_{-\frac{1}{2}}v^{\frac{k}{2},\frac{k}{2}}$, $(e - f)_{-\frac{1}{2}}v^{\frac{k}{2},\frac{k}{2}}\}$ form a basis of the top level of $\overline{L(k,\frac{k}{2})}^{\sigma_1,-}$.
Moreover, notice that $h(-1)\mathbbm{1} \in K^{\langle \sigma_{1} \rangle}\setminus L(k,0)^{K}$ and
\[ (h(-1)\mathbbm{1})_{0}(v^{\frac{k}{2},\frac{k}{2}-1} \pm f_{-\frac{1}{2}}v^{\frac{k}{2},\frac{k}{2}}) = 2(v^{\frac{k}{2},\frac{k}{2}-1} \mp f_{-\frac{1}{2}}v^{\frac{k}{2},\frac{k}{2}}).   \]
This implies that $(e + f)_{-\frac{1}{2}}v^{\frac{k}{2},\frac{k}{2}}$ and $(e - f)_{-\frac{1}{2}}v^{\frac{k}{2},\frac{k}{2}}$ must belong to different direct summands of the decomposition of $\overline{L(k,\frac{k}{2})}^{\sigma_1,-}$ into $L(k,0)^{K}$-modules. The lemma is proved.

\end{proof}

\begin{lem}  \label{twisted decomposition}
For $0 \leqslant i <  \frac{k}{2}$, we have
\begin{equation}  \label{twisted V^G dep.1}
\overline{L(k,i)}^{\sigma_r} = \overline{L(k,i)}^{\sigma_r,+} \oplus \overline{L(k,i)}^{\sigma_r,-}
\end{equation}
for each $r = 1$, $2$, $3$, where $\overline{L(k,i)}^{\sigma_r,+} ($resp. $\overline{L(k,i)}^{\sigma_r,-} )$ is the irreducible $L(k,0)^{K}$-module generated by the lowest weight vector $v^{r,i,i} ($resp. $f^{(r)}_{-\frac{1}{2}}v^{r,i,i} )$ with the lowest weight $\frac{i(i-k)}{4(k+2)} + \frac{k}{16} ($resp. $\frac{i(i-k)}{4(k+2)} + \frac{k}{16} + \frac{1}{2} )$.

\vskip 0.2cm
If $k \in 2\mathbb{Z}_{+}$, we have
\begin{equation}   \label{twisted V^G dep.2}
\overline{L(k,\frac{k}{2})}^{\sigma_r}=\bigoplus_{j=0}^{j=3}\overline{L(k,\frac{k}{2})}^{\sigma_r,(j)}
\end{equation}
for each $r = 1$, $2$, $3$, where $\overline{L(k,\frac{k}{2})}^{\sigma_r,(0)} ($resp. $\overline{L(k,\frac{k}{2})}^{\sigma_r,(1)}$, $\overline{L(k,\frac{k}{2})}^{\sigma_1,(2)}$, $\overline{L(k,\frac{k}{2})}^{\sigma_r,(3)} )$ is the irreducible $L(k,0)^{K}$-module generated by the lowest weight vector $v^{r,\frac{k}{2},\frac{k}{2}} ($resp. $h^{(r)}(-1)v^{r,\frac{k}{2},\frac{k}{2}}$, $(e^{(r)} + f^{(r)})_{-\frac{1}{2}}v^{r,\frac{k}{2},\frac{k}{2}}$, $(e^{(r)} - f^{(r)})_{-\frac{1}{2}}v^{r,\frac{k}{2},\frac{k}{2}} )$ with the lowest weight $\frac{k}{8(k+2)} ($resp. $\frac{k}{8(k+2)}+1$, $\frac{k}{8(k+2)} + \frac{1}{2}$, $\frac{k}{8(k+2)} + \frac{1}{2} )$.
\end{lem}
\begin{proof}
For $r=1$ and $i \neq \frac{k}{2}$, $K_{\overline{L(k,i)}^{\sigma_1}} = \{ \sigma_0, \sigma_1\}$  from Lemma \ref{twisted G_M subgroup}. Then the twisted group algebra $\mathbb{C}^{\alpha_{\overline{L(k,i)}^{\sigma_1}}}[K_{\overline{L(k,i)}^{\sigma_1}}]$ is a commutative semisimple associative algebra which has only two irreducible modules of dimension one. The decomposition of $\overline{L(k,i)}^{\sigma_1}$ into irreducible $L(k,0)^{K}$-modules is exactly \eqref{Z_2 decomposition eq2}.
And $\overline{L(k,i)}^{\sigma_1,+}$, $\overline{L(k,i)}^{\sigma_1,-}$ are inequivalent irreducible $L(k,0)^{K}$-modules due to Theorem \ref{DRX17 thm1}. Thus \eqref{twisted V^G dep.1} holds for $r=1$ and the lowest weight vectors with their lowest weights of these irreducible $L(K,0)^{K}$-modules are listed in Theorem \ref{Z_2 decomposition}.

If $k \in 2\mathbb{Z}_{+}$, $K_{\overline{L(k,\frac{k}{2})}^{\sigma_1}} = K$ from Lemma \ref{twisted G_M subgroup}. The $K$-orbit $\overline{L(k,\frac{k}{2})}^{\sigma_1} \circ K$ of $\overline{L(k,\frac{k}{2})}^{\sigma_1}$ only contains itself, and the map $\sigma_r \mapsto \phi(\sigma_r) ( r = 0, 1, 2, 3 )$ defined in Lemma \ref{twisted G_M subgroup} gives a projective representation of $K$ on $\overline{L(k,\frac{k}{2})}^{\sigma_1}$.
Let $\alpha_{\overline{L(k,\frac{k}{2})}^{\sigma_1}}$ be the corresponding $2$-cocycle in $C^2(K, \mathbb{C}^*)$.
Then $\overline{L(k,\frac{k}{2})}^{\sigma_1}$ is a module for the twisted group algebra $\mathbb{C}^{\alpha_{\overline{L(k,\frac{k}{2})}^{\sigma_1}}}[K]$. And $\mathbb{C}^{\alpha_{\overline{L(k,\frac{k}{2})}^{\sigma_1}}}[K]$ is a commutative semisimple associative algebra which has four irreducible modules of dimension one.
Let $\overline{L(k,\frac{k}{2})}^{\sigma_1,(0)}$ (resp. $\overline{L(k,\frac{k}{2})}^{\sigma_1,(1)}$, $\overline{L(k,\frac{k}{2})}^{\sigma_1,(2)}$, $\overline{L(k,\frac{k}{2})}^{\sigma_1,(3)}$ ) be the irreducible $L(k,0)^{K}$-module generated by the lowest weight vector $v^{\frac{k}{2},\frac{k}{2}}$ (resp. $h(-1)v^{\frac{k}{2},\frac{k}{2}}$, $(e + f)_{-\frac{1}{2}}v^{\frac{k}{2},\frac{k}{2}}$, $(e - f)_{-\frac{1}{2}}v^{\frac{k}{2},\frac{k}{2}}$ ) with the lowest weight $\frac{k}{8(k+2)}$ (resp. $\frac{k}{8(k+2)}+1$, $\frac{k}{8(k+2)} + \frac{1}{2}$, $\frac{k}{8(k+2)} + \frac{1}{2}$ ).
Then $\overline{L(k,\frac{k}{2})}^{\sigma_1,+} = \overline{L(k,\frac{k}{2})}^{\sigma_1,(0)} \oplus \overline{L(k,\frac{k}{2})}^{\sigma_1,(1)}$ and $\overline{L(k,\frac{k}{2})}^{\sigma_1,-} = \overline{L(k,\frac{k}{2})}^{\sigma_1,(2)} \oplus \overline{L(k,\frac{k}{2})}^{\sigma_1,(3)}$. As a result, \eqref{twisted V^G dep.2} holds for $r=1$.

In almost exactly the same way, the proof for the case $r = 2$ and $r = 3$ are similar to the case $r = 1$. Therefore, we obtain our result as desired.
\end{proof}

To sum up, if $k \in 2\mathbb{Z}_{+}-1$, there are $k+1$ irreducible $L(k,0)^{K}$-modules coming from $\sigma_r$-twisted $L(k,0)$-modules $\overline{L(k,i)}^{\sigma_r}, 0 \leqslant i\leqslant k$ for each $r = 1$, $2$ or $3$ up to isomorphism.
If $k \in 2\mathbb{Z}_{+}$, there are $k+4$ irreducible $L(k,0)^{K}$-modules coming from $\sigma_r$-twisted $L(k,0)$-modules $\overline{L(k,i)}^{\sigma_r}, 0 \leqslant i\leqslant k$ for each $r = 1$, $2$ or $3$ up to isomorphism.

\vskip0.3cm
From the above discussion and Theorem \ref{DRX17 thm3}, we obtain one of the main results of this paper.

\begin{thm}\label{main1}
(1) If $k \in 2\mathbb{Z}_{+}-1$, there are $\frac{11(k+1)}{2}$ inequivalent irreducible $L(k,0)^{K}$-modules as follows:
$$
L(k, i)^{+}, \ L(k,j)^{(l)}, \ 1\leqslant i\leqslant k, \ i\in 2\Z+1, \ 0\leqslant j\leqslant k, \ j\in 2\Z, \ 0\leqslant l\leqslant 3,
$$
$$
\overline{L(k,i)}^{\sigma_r,+}, \ \overline{L(k,i)}^{\sigma_r,-}, \ 0\leqslant i\leqslant \frac{k-1}{2}, \ r=1,2,3.
$$
(2) If $k \in 2\mathbb{Z}_{+}$, there are $\frac{11k+32}{2}$ inequivalent irreducible $L(k,0)^{K}$-modules as follows:
$$
L(k, i)^{+}, \ L(k,j)^{(l)}, \ 1\leqslant i\leqslant k, \ i\in 2\Z+1, \ 0\leqslant j\leqslant k, \ j\in 2\Z, \ 0\leqslant l\leqslant 3,
$$
$$
\overline{L(k,i)}^{\sigma_r,+}, \ \overline{L(k,i)}^{\sigma_r,-}, \ \overline{L(k,\frac{k}{2})}^{{\sigma_r}, (l)}, \ 0\leqslant i\leqslant \frac{k}{2}-1, \ r=1,2,3, \ 0\leqslant l\leqslant 3.
$$
\end{thm}
\begin{proof}
From Theorem \ref{DRX17 thm3}, we know that any irreducible $L(k,0)^{K}$-module occurs in an irreducible $\sigma_r$-twisted $L(k,0)$-module for some $r \in \{0, 1, 2, 3\}$. In more detail, $L(k,i)^+ ( 0 \leqslant i \leqslant k, i \in 2\mathbb{Z}+1 )$, $L(k,i)^{(j)} ( 0 \leqslant i \leqslant k, i \in 2\mathbb{Z}, j \in \{0, 1, 2, 3 \} )$, $\overline{L(k,i)}^{\sigma_r, \pm} ( 0 \leqslant i < \frac{k}{2}, r \in \{1, 2, 3\} )$ and $\overline{L(k,\frac{k}{2})}^{\sigma_r,(j)} ( r \in \{1, 2, 3\}, j \in \{0, 1, 2, 3\}, k \in 2\mathbb{Z} )$ give a complete list of inequivalent irreducible $L(k,0)^{K}$-modules. The lowest weight vectors with their lowest weights of these irreducible $L(k,0)^{K}$-modules are listed in Theorem \ref{V decomposition}, Lemma \ref{irr decomposition} and Lemma \ref{twisted decomposition}.
\end{proof}

\begin{rmk}
For $k=1$, the orbifold vertex operator algebra $L(1,0)^{K}$ can be realized as the fixed point subalgebra $V_{\Z\be}^+$ of the lattice vertex operator algebra $V_{\Z\be}$ associated to the positive definite even lattice $\Z\be$ with $(\be, \be) = 8$ under an automorphism of $V_{\Z\be}$ lifting the $-1$ isometry of ${\Z\be}$ \cite{DG98}. Moreover, the representations $V_{\Z\be}^+$ have been studied in \cite{DJ13} and \cite{DN99}. It turns out there are $11$ inequivalent irreducible $V_{\Z\be}^+$-modules which have been listed  explicitly in \cite{DJ13}.
\end{rmk}

\begin{rmk}
Notice that for $k=2$, $L(2,0)=V_{\Z\gamma}\otimes L_{Vir}(\frac{1}{2},0)\oplus V_{\Z\gamma + \frac{1}{2}\gamma}\otimes L_{Vir}(\frac{1}{2},\frac{1}{2})$ and  $L(2,0)^{\langle \sigma_1 \rangle}=V_{\Z\gamma}\otimes L_{Vir}(\frac{1}{2},0)$, where  $(\gm, \gm) = 4$, and $L_{Vir}(\frac{1}{2}, 0)$ is the simple vertex operator algebra associated to the Virasoro algebra with central charge $\frac{1}{2}$. Then it is easy to see that $L(2,0)^{K}$ is isomorphic to $V_{\Z\gm}^+ \otimes L_{Vir}(\frac{1}{2}, 0)$  Using the fact that $V_{\Z\gm}^+$ is isomorphic to $L_{Vir}(\frac{1}{2},0)^{\otimes 2}$ \cite{DGH98}, we have $L_{\widehat{\mathfrak{sl}_2}}(2,0)^{K} \cong L_{Vir}(\frac{1}{2},0)^{\otimes 3}$. Then there are exactly $27$ inequivalent irreducible $L_{Vir}(\frac{1}{2},0)^{\otimes 3}$-modules.

\end{rmk}

Recall from Theorem \ref{jl1} that for $m\in 2\Z_{\geqslant 2}+1$, $C_{{L_{\widehat{\mathfrak{so}_m}}(1,0)}^{\otimes 3}}({L_{\widehat{\mathfrak{so}_m}}(3,0)})\cong L(2m, 0)^{K}$, and $C_{{L_{\widehat{\mathfrak{so}_m}}(1,0)}^{\otimes 3}}({L_{\widehat{\mathfrak{so}_m}}(3,0)})\cong (L(2m, 0)\oplus L(2m, 2m))^{K}$ for $m\in 2\Z_{\geqslant 2}$. We are now in a position to give the other main result of this paper.
\begin{thm} (1) For $m\in 2\Z_{\geqslant 2}+1$, $C_{{L_{\widehat{\mathfrak{so}_m}}(1,0)}^{\otimes 3}}({L_{\widehat{\mathfrak{so}_m}}(3,0)})$ has $11m+16$ inequivalent irreducible modules as follows:
$$
L(2m, i)^{+}, \ L(2m,j)^{(l)}, \ 1\leqslant i\leqslant 2m, \ i\in 2\Z+1, \ 0\leqslant j\leqslant 2m, \ j\in 2\Z, \ 0\leqslant l\leqslant 3,
$$
$$
\overline{L(2m,i)}^{\sigma_r,+}, \ \overline{L(2m,i)}^{\sigma_r,-}, \ \overline{L(2m, m)}^{{\sigma_r}, (l)}, \ 0\leqslant i\leqslant m-1, \ r=1,2,3, \ 0\leqslant l\leqslant 3.
$$

(2) For $m\in 2\Z_{\geqslant 2}$, $C_{{L_{\widehat{\mathfrak{so}_m}}(1,0)}^{\otimes 3}}({L_{\widehat{\mathfrak{so}_m}}(3,0)})\cong L(2m,0)^K\oplus L(2m,2m)^K$ has $8m+32$ inequivalent irreducible modules as follows:
$$
L(2m, i)^{(j)}\oplus L(2m,2m-i)^{(j)},$$ \
where $0\leqslant i\leqslant m-2, \ i\in 2\Z$, $0\leqslant j\leqslant 3$  with the ``split'' type ones:
$$
(L(2m,m)^{(j)})^{\pm}, \ (\overline{L(2m,i)}^{\sigma_r,+})^{\pm}, \ (\overline{L(2m,i)}^{\sigma_r,-})^{\pm}, \ (\overline{L(2m, m)}^{{\sigma_r}, (j)})^{\pm},$$
 where $0\leqslant i\leqslant m-2, i\in 2\Z$, \ $r=1,2,3, \ 0\leqslant j\leqslant 3.$
\end{thm}
\begin{proof}
(1) follows from Theorem \ref{main1} and Theorem \ref{jl1}.
We now prove (2). Since $C_{{L_{\widehat{\mathfrak{so}_m}}(1,0)}^{\otimes 3}}({L_{\widehat{\mathfrak{so}_m}}(3,0)})\cong (L(2m, 0)\oplus L(2m, 2m))^{K}$, it follows that the irreducible modules of $C_{{L_{\widehat{\mathfrak{so}_m}}(1,0)}^{\otimes 3}}({L_{\widehat{\mathfrak{so}_m}}(3,0)})$ occurs in the irreducible $\sigma_r$-twisted  modules of  $L(2m, 0)\oplus L(2m, 2m)$, $0\leqslant r\leqslant 3$.
 Denote
 $$\widetilde{V}^{\sigma_1}=L(2m, 0)^{\langle \sigma_{1} \rangle}\oplus L(2m, 2m)^{\langle \sigma_{1} \rangle}.$$

  Notice that $m\in 2\Z_{+}$, and for $1\leqslant i\leqslant m-1, i\in 2\Z+1$, the lowest weights of the $L(2m,0)^{\langle \sigma_{1} \rangle}$-module $L(2m,i)^{\sigma_1,+}$ is $\frac{i(i+2)}{2(2m+4)}$. Then it is easy to check that
  $$
  \frac{i(i+2)}{2(2m+4)}-\frac{(2m-i)(2m-i+2)}{2(2m+4)}\in \Z+\frac{1}{2}.
  $$
  By (\ref{fusion.untwist1.}) in Theorem \ref{fusion-aff} about  the fusion rules of the vertex operator algebra $L(2m,0)^{\langle \sigma_{1} \rangle}$, we have
  $$
  L(2m,2m)^{\sigma_1,+}\boxtimes L(2m,i)^{\sigma_1,\pm}=L(2m,2m-i)^{\sigma_1,\mp}.
 $$
 Then by Theorem \ref{sc}, for $1\leqslant i\leqslant m-1, i\in 2\Z+1$, $L(2m,i)^{\sigma_1,\pm}\oplus L(2m,2m-i)^{\sigma_1,\mp}$ are not modules of $\widetilde{V}^{\sigma_1}$. Similarly, by (\ref{fusion.twist1}),
 $$
 L(2m,2m)^{\sigma_1,+}\boxtimes \overline{L(2m,i)}^{\sigma_1,\pm}=\overline{L(2m,2m-i)}^{\sigma_1,\mp},
 $$
 for $1\leqslant i\leqslant m-1, i\in 2\Z+1$. Then we also deduce that
 $$
 \overline{L(2m,i)}^{\sigma_1,\pm}\oplus \overline{L(2m,2m-i)}^{\sigma_1,\mp}, \ 1\leqslant i\leqslant m-1, \  i\in 2\Z+1
 $$
 are not modules of $\widetilde{V}^{\sigma_1}$. By Theorem \ref{fusion-aff}, for $0\leqslant i\leqslant m$, $i\in 2\Z$,
 $$
 L(2m,2m)^{\sigma_1,+}\boxtimes L(2m,i)^{\sigma_1,\pm}=L(2m,2m-i)^{\sigma_1,\pm},$$
 $$
  L(2m,2m)^{\sigma_r,+}\boxtimes \overline{L(2m,i)}^{\sigma_r,\pm}=\overline{L(2m,2m-i)}^{\sigma_r,\pm}. \
 $$
 Then we  deduce that
 the ``nonsplit'' type modules
 $$
 L(2m, i)^{\sigma_1,\pm}\oplus L(2m, 2m-i)^{\sigma_1,\pm}, \ \overline{L(2m, i)}^{\sigma_1,\pm}\oplus \overline{L(2m, 2m-i)}^{\sigma_1,\pm},
 $$
 where $ 0\leqslant i\leqslant m-2, \ i\in 2\Z$, and the ``split'' type modules
 $$
 (L(2m, m)^{\sigma_1,+})^{\pm}, \ (L(2m, m)^{\sigma_1,-})^{\pm}, \ (\overline{L(2m, m)}^{\sigma_1,+})^{\pm}, \ (\overline{L(2m, m)}^{\sigma_1,-})^{\pm}
 $$
exhaust all the irreducible modules of $\widetilde{V}^{\sigma_{1}}$. By (\ref{esc}) and (\ref{phi 2}), all the above $\widetilde{V}^{\sigma_1}$-modules are stable under $\sigma_2$. Then by  \cite{DLM00}, \cite{DRX17}, and \cite{MT04} we deduce that $C_{{L_{\widehat{\mathfrak{so}_m}}(1,0)}^{\otimes 3}}({L_{\widehat{\mathfrak{so}_m}}(3,0)}) \cong (L(2m, 0)\oplus L(2m, 2m))^{K}$ has exactly $8m+32$ irreducible modules.

On the other hand, By Theorem \ref{sc},  for $r=1,2,3$, $1\leqslant i\leqslant m-1, i\in 2\Z+1$,
$$
 L(2m,i)^{+}\oplus  L(2m,2m-i)^{+},
 \
 \overline{L(2m,i)}^{\sigma_r,\pm}\oplus \overline{L(2m,2m-i)}^{\sigma_r,\mp},
 $$
can not be modules of the simple extension $L(2m,0)^K\oplus L(2m,2m)^K$. Then by (2) of Theorem \ref{main1} we have (2).

\end{proof}
\begin{rmk} For $m=6$,  we have
$$
C_{{L_{\widehat{\mathfrak{so}_6}}(1,0)}^{\otimes 3}}({L_{\widehat{\mathfrak{so}_6}}(3,0)})\cong (L(12, 0)\oplus L(12, 12))^{K}.
$$
On the other hand, from \cite{JLin} and \cite{Lam14} and the fact that $\frak{so_6}(\C)\cong \frak{sl}_{4}(\C)$, we have
$$C_{{L_{\widehat{\mathfrak{so}_6}}(1,0)}^{\otimes 3}}({L_{\widehat{\mathfrak{so}_6}}(3,0)})\cong K(\frak{sl}_3,4),$$
where $K(\frak{sl}_3, 4)$ is the parafermion vertex operator algebra associated with $\frak{sl}_{3}(\C)$ and the level $4$ \cite{ADJR18}, \cite{DLWY}, \cite{DLY}. Then
we have
$$(L(12, 0)\oplus L(12, 12))^{K}\cong K(\frak{sl}_3,4).$$
By \cite{ADJR18}, $K(\frak{sl}_3, 4)$ has exactly $80$ inequivalent irreducible modules.
\end{rmk}

\end{document}